\documentclass[a4paper, 11pt]{amsart}
\usepackage{amssymb,amsmath,comment,mathrsfs}
\usepackage[top=30truemm,bottom=30truemm,left=25truemm,right=25truemm]{geometry}

\newcommand{\Z}{\mathbb Z}
\newcommand{\Q}{\mathbb Q}
\newcommand{\C}{\mathbb C}
\newcommand{\F}{\mathbb F}
\newcommand{\Gal}{\mathrm{Gal}}

\theoremstyle{plain}
\newtheorem{thm}{Theorem}[section]
\newtheorem*{thm*}{Theorem}

\newtheorem{prop}[thm]{Proposition}
\newtheorem{lem}[thm]{Lemma}
\newtheorem{cor}[thm]{Corollary}

\theoremstyle{definition}

\newtheorem{remark}[thm]{Remark}

\newtheorem*{acknowledgment}{Acknowledgment}

\begin{document}
\title[Lifts of multiplication matrices of Gaussian periods]
{Davenport and Hasse's theorems and lifts of multiplication matrices of Gaussian periods}

\author[A. Hoshi]{Akinari Hoshi}
\address[A. Hoshi]{Department of Mathematics, Niigata University, Niigata 950-2181, Japan}
\email{hoshi@math.sc.niigata-u.ac.jp}

\author[K. Kanai]{Kazuki Kanai}
\address[K. Kanai]{Graduate School of Science and Technology, Niigata University, Niigata 950-2181, Japan}
\email{kanai@m.sc.niigata-u.ac.jp}

\footnote[0]{\textit{$2020$ Mathematics Subject
Classification}. 11D09, 11R18, 11T22, 11T24.}
%11D09 Quadratic and bilinear Diophantine equations
%11R18 Cyclotomic extensions
%11T22 Cyclotomy
%11T23 Exponential sums
%11T24 Other character sums and Gauss sums
%Cyclotomy, Dickson's system
\footnote[0]{This work was partially supported by JSPS KAKENHI Grant Number 19K03418.}

%%%%%%%%%%%%%%%%%%%%%%%%%%%%%%%%%%%%%%%%%

\begin{abstract}
Let $e \geq 2$ be an integer, $p^r$ be a prime power with 
$p^r \equiv 1\ ({\rm mod}\ e)$ 
and $\eta_r(i)$ be Gaussian periods of degree $e$ for $\F_{p^r}$. 
By the dual form of Davenport and Hasse's lifting theorem on Gauss sums, 
we establish lifts of the multiplication matrices of the Gaussian periods 
$\eta_r(0),\ldots,\eta_r(e-1)$ which are defined by F. Thaine. 
We also give some examples of the explicit lifts for prime degree $e$ 
with $3\leq e\leq 23$ which also illustrate relations among lifts of 
Jacobi sums, Gaussian periods and multiplication matrices of Gaussian periods. 
\end{abstract}

\maketitle
%%%%%%%%%%%%%%%%%%%%%%%%%%%%%%%%%%%%%%%%%
\tableofcontents

%------------------------------------S1
\section{Introduction}\label{S1}

Let $e \geq 2$ be an integer and 
$p^r$ be a prime power with $p^r \equiv 1\ ({\rm mod}\ e)$. 
Write $p^r=ef+1$. 
Let $\F_{p^{r}}$ be the finite field of $p^r$ elements 
and $\gamma$ be a fixed generator of 
$\F_{p^r}^{\times}=\F_{p^{r}}\setminus \{0\}$. 
Let $\zeta_n=e^{2\pi i/n}$ be an 
%the 
$n$-th root of unity. 
For $0 \leq i \leq e-1$, 
{\it Gaussian periods $\eta_r(i)$ of degree $e$ for $\F_{p^r}$} 
are defined by
\begin{align*}
\eta_r(i) := \sum^{f-1}_{j=0}\zeta_p^{\mathrm{Tr}(\gamma^{ej+i})}
\end{align*}
where $\mathrm{Tr}$ is the trace map $\mathrm{Tr} : \F_{p^r} \to \F_p$, 
and {\it the period polynomial $P_{e,r}(X)$ of degree $e$ for $\F_{p^r}$} 
is defined by $P_{e,r}(X)= \prod^{e-1}_{i=0}\left(X-\eta_r(i)\right)\in\Z[X]$. 
Note that $\eta_{r}(i)$ depends on the choice of $\gamma$. 

We recall notion of Jacobi sums, Gauss sums and cyclotomic numbers 
and their relations (see Berndt, Evans and Williams \cite{BEW98}). 
For a nontrivial character $\psi$ on $\F_{p^r}^{\times}$ and 
the trivial character $\varepsilon$ on $\F_{p^r}^{\times}$, 
we extend them to $\F_{p^r}$ by setting $\psi(0)=0$ and $\varepsilon(0)=1$. 
Let $\psi_1$, $\psi_2$ be characters on $\F_{p^r}$. 
{\it Jacobi sums $J_r(\psi_1,\psi_2)$ and $J_r^{\ast}(\psi_1,\psi_2)$ for $\F_{p^r}$} are defined by
\[
J_r(\psi_1,\psi_2):=\sum_{\alpha \in \F_{p^r}}\psi_1(\alpha)\psi_2(1-\alpha)
\quad \mbox{and}\quad
J_r^{\ast}(\psi_1,\psi_2):=\sum_{\substack{\alpha \in \F_{p^r}\\ \alpha \neq 0,1}}\psi_1(\alpha)\psi_2(1-\alpha).
\]
If $\psi_1, \psi_2 \neq \varepsilon$, then 
$J_{r}(\psi_1,\psi_2)=J_{r}^{\ast}(\psi_1,\psi_2)$ 
although $J_{r}(\varepsilon, \psi_2)=J_{r}(\psi_1, \varepsilon)=0$, 
$J_{r}(\varepsilon, \varepsilon)=p^r$ and 
$J_{r}^{\ast}(\varepsilon,\psi_2)=J_{r}^{\ast}(\psi_1, \varepsilon)=-1$, 
$J_{r}^{\ast}(\varepsilon, \varepsilon)=p^r-2$.
{\it Gauss sums $G_{r}(\psi_1)$ and $G_{r}^{\ast}(\psi_1)$ for $\F_{p^r}$} 
are defined to be
\[
G_{r}(\psi_1):=\sum_{\alpha \in \F_{p^r}}\psi_1(\alpha)\zeta_p^{{\rm Tr}(\alpha)}
\quad \mbox{and}\quad
G_{r}^{\ast}(\psi_1):=\sum_{\alpha \in \F_{p^r}^{\times}}\psi_1(\alpha)\zeta_p^{{\rm Tr}(\alpha)}.
\]
If $\psi_1 \neq \varepsilon$, then $G_{r}(\psi_1)=G_{r}^{\ast}(\psi_1)$ 
although $G_{r}(\varepsilon)=0$ and $G_{r}^{\ast}(\varepsilon)=-1$. 
We have the well-known relations
\[
J_r(\psi_1,\psi_2)=\frac{G_{r}(\psi_1)G_{r}(\psi_2)}{G_{r}(\psi_1\psi_2)}
\quad \mbox{and}\quad
J_r^{\ast}(\psi_1,\psi_2)=\frac{G_{r}^{\ast}(\psi_1)G_{r}^{\ast}(\psi_2)}{G_{r}^{\ast}(\psi_1\psi_2)}
\]
whenever $\psi_1\psi_2 \neq \varepsilon$.

The following theorem is the Davenport and Hasse's lifting theorem 
(see also Weil \cite[pages 503--505]{Wei49}, \cite[page 360, Theorem 11.5.2]{BEW98}):
\begin{thm}[{Davenport and Hasse \cite[Relation (0.8)]{DH35}}]\label{t1.1}
Let $e \geq 2$ be an integer, $p^r$ be a prime power  
%with $p^r\equiv 1\ ({\rm mod}\ e)$ 
and $\psi$ be a nontrivial character on $\F_{p^r}$. 
Then, for any integer $n \geq 1$, we have
\[
G_{nr}(\psi^{\prime})=(-1)^{n-1}G_r(\psi)^n
\]
where $\psi^{\prime}$ is the lift of $\psi$ 
from $\F_{p^r}$ to $\F_{p^{nr}}$ defined by 
$\psi^{\prime}(\alpha) = \psi(\mathrm{Nr}(\alpha))$ 
and $\mathrm{Nr}$ is the norm map. 
In particular, if $\psi_1$, $\psi_2$ and $\psi_1\psi_2$ are 
nontrivial characters on $\F_{p^r}$, then we have 
\[
J_{nr}(\psi_1^{\prime}, \psi_2^{\prime})=(-1)^{n-1}J_r(\psi_1, \psi_2)^n.
\]
\end{thm}

{}From now on, we take the character $\chi$ of order $e$ 
on $\F_{p^r}$ with $\chi(\gamma)=\zeta_e$ and $\chi(0)=0$ 
where $\langle\gamma\rangle=\F_{p^r}^\times$. 
For $0 \leq i,j \leq e-1$, 
we simply write 
\[
J_r(i,j):= J_r(\chi^i,\chi^j)
\quad \mbox{and}\quad
J_r^{\ast}(i,j):=J_r^{\ast}(\chi^i,\chi^j).
\]
We have
\[
J_r(i,j) = J_r^{\ast}(i,j) +\delta_{i,0}+\delta_{0,j}
\]
where $\delta_{i,j}$ is Kronecker's delta. 

\begin{remark}
{\rm 
(1) If we adopt the another manner $\varepsilon(0)=0$, then we have 
$J_r(i,j)=J_r^{\ast}(i,j)$ (e.g. Myerson \cite{Mye81}, 
Parnami, Agrawal and Rajwade \cite{PAR82}, Katre and Rajwade \cite{KR85a}).\\
(2) We may take another Jacobi sum
\begin{align*}
J_r^{\prime}(i,j)=\sum_{\alpha \in \F_{p^r}}\chi^i(\alpha)\chi^j(1+\alpha)
\end{align*}
which satisfies $J_r^{\prime}(i,j) = \chi(-1)^i J_r(i,j)=(-1)^{fi}J_r(i,j)$ 
(see e.g. \cite{PAR82}, \cite{KR85a}).
}
\end{remark}

For $0 \leq i,j \leq e-1$, 
{\it cyclotomic numbers ${\rm Cyc}_r(i,j)$ of order $e$ for $\F_{p^r}$} 
are defined by
\begin{align*}
{\rm Cyc}_r(i,j) :=\# \{ (v_1, v_2) \mid 0 \leq v_1, v_2 \leq f-1, \ 1+\gamma^{ev_1+i} \equiv \gamma^{ev_2+j} \ ({\rm mod}\ p^r) \}.
\end{align*}
We have the following well-known 
relations between cyclotomic numbers ${\rm Cyc}_r(a,b)$ 
and Jacobi sums $J_r^{\ast}(i,j)$ 
(see \cite[page 79, Theorem 2.5.1]{BEW98}):
\[
\sum_{i=0}^{e-1} \sum_{j=0}^{e-1} (-1)^{fi} \zeta_e^{-(ai+bj)}J_r^{\ast}(i,j)=e^2{\rm Cyc}_r(a,b) 
\]
and
\[
(-1)^{fi}\sum_{a=0}^{e-1} \sum_{b=0}^{e-1} {\rm Cyc}_r(a,b) \zeta_e^{ai+bj}=J_r^{\ast}(i,j). 
\]
Note that both ${\rm Cyc}_r(a,b)$ and $J_r^{\ast}(i,j)$ 
depend on the choice of the fixed generator $\gamma$ of $\F_{p^{r}}^{\times}$. 

We see that the product of the Gaussian periods is represented 
by a linear combination of the Gaussian periods again 
and these coefficients are given in terms of the cyclotomic numbers 
(see \cite[page 328, Lemma 10.10.2,  page 437, Exercise 12.23]{BEW98}):
\begin{align}\label{eq1}
\eta_r(m)\eta_r(m+i)=\sum^{e-1}_{j=0} ({\rm Cyc}_r(i,j) -D_{i}f)\eta_r(m+j)
\end{align}
where $D_i = \delta_{0,i}$ (resp. $\delta_{e/2 ,i}$) if $f$ is even (resp. odd). 
It follows that the Gaussian periods are the eigenvalues of the 
$e \times e$ matrix $C_r := [{\rm Cyc}_r(i,j) - D_i f]_{0\leq i,j\leq e-1}$ 
called {\it the multiplication matrix of $\eta_r(0),\ldots,\eta_r(e-1)$} 
(see Section \ref{S2}). 
Hence the period polynomial $P_{e,r}(X)$ can be obtained 
as the characteristic polynomial ${\rm Char}_X(C_r)$ of 
the multiplication matrix $C_r$. 

F. Thaine investigated various properties and characterizations of 
Gaussian periods, cyclotomic numbers and Jacobi sums 
with applications to the construction of cyclic polynomials 
in the series of the papers \cite{Tha96}, \cite{Tha99}, \cite{Tha00}, \cite{Tha01}, \cite{Tha04}, \cite{Tha08} (see also 
Lehmer \cite{Leh88}, 
Schoof and Washington \cite{SW88}, 
Hashimoto and Hoshi \cite{HH05a}, \cite{HH05b}). 

According to Thaine \cite[Section 2]{Tha04}, 
for two $e\times e$ matrices $A=[a_{i,j}]_{0\leq i,j\leq e-1}$, 
$B=[b_{i,j}]_{0\leq i,j\leq e-1}$ 
and $d\in(\Z/e\Z)\setminus\{0\}$, 
we define {\it the $d$-composition $A \overset{d}{\ast} B$ of $A$ and $B$} as 
\[
A \overset{d}{\ast} B:= \left[\sum^{e-1}_{s=0}\sum^{e-1}_{t=0}a_{s,t}b_{ds+i,dt+j}\right]_{0\leq i,j\leq e-1}.
\]
For the multiplication matrix $C_1$ (resp. $C^{\prime}_1$) 
of Gaussian periods $\eta_1(0),\ldots,\eta_1(e-1)$ 
(resp. $\eta^{\prime}_1(0),\ldots,\eta^{\prime}_1(e-1)$) of degree $e$ for $\F_{p^1}$, 
the $d$-composition 
$C_1\overset{d}{\ast} C^{\prime}_1$ 
gives the multiplication matrix of $\theta_{d,0},\ldots,\theta_{d,e-1}$ 
where 
$\theta_{d,i}=\sum_{s=0}^{e-1}\eta_1(s)\eta^{\prime}_1(ds+i)$. 
Hence we obtain the cyclic polynomial 
${\rm Char}_X(C_1\overset{d}{\ast} C^{\prime}_1)$ 
which gives an intermediate cyclic field $\Q(\theta_{d,i})$ of degree $e$ 
in the composite bicyclic field $\Q(\eta_1(0),\eta^{\prime}_1(0))
\subset\Q(\zeta_p,\zeta_q)$ with $p\neq q$ 
(see Section \ref{S2} and Section \ref{S3}). 

In Section \ref{S4}, we regard the Gaussian periods 
$\eta_r(i)$ of degree $e$ for $\F_{p^r}$ as the function $\eta_r:\Z/e\Z\rightarrow \C$, 
$i\mapsto \eta_r(i)$ 
and the Gauss sums $G^\ast_r(\chi)$ for $\F_{p^r}$ as the function 
$G^\ast_r:\widehat{\Z/e\Z}\rightarrow \C$, $\chi\mapsto G^\ast_r(\chi)$. 
Then we find that they 
are each other's finite Fourier transform with some twist 
(see Lemma \ref{l4.1} in Section \ref{S4}) and we have:
\begin{thm}[{Davenport and Hasse's lifting theorem: the dual form}]\label{t1.3}
Let $e \geq 2$ be an 
integer and $p^r$ be a prime power with 
$p^r\equiv 1\ ({\rm mod}\ e)$. 
We regard the Gaussian periods 
$\eta_r(i)$ of degree $e$ for $\F_{p^r}$ as the function $\eta_r:\Z/e\Z\rightarrow \C$, 
$i\mapsto \eta_r(i)$. 
Then, for any integer $n \geq 1$, we have 
\[
\eta_{nr}(i) = (-1)^{n-1}\eta_r^{(n)}(i)\quad \mbox{for}\quad 0\leq i \leq e-1
\]
where 
\begin{align*}
\eta_{r}^{(n)}(i)=\sum_{k_1+\cdots+k_n\equiv\,i\,({\rm mod}\,e)\atop 0\leq k_1,\ldots,k_n\leq e-1}\eta_{r}(k_1)\cdots\eta_{r}(k_n)
\end{align*} 
is the $n$-fold product of $\eta_{r}$ with respect to the convolution product. 
\end{thm}

By using Theorem \ref{t1.3}, 
we get our main theorem which gives lifts of the multiplication matrix $C_r$ 
of the Gaussian periods $\eta_r(0),\ldots,\eta_r(e-1)$ of degree $e$ for $\F_{p^{r}}$
via Thaine's $(-1)$-composition $\overset{-1}{\ast}$: 
\begin{thm}\label{t1.4}
Let $e \geq 2$ be an  
integer and $p^r$ be a prime power with $p^r\equiv 1\ ({\rm mod}\ e)$.
Let $C_r = [{\rm Cyc}_r(i,j) - D_i f]_{0\leq i,j\leq e-1}$ be the multiplication matrix of the Gaussian periods $\eta_r(0),\ldots,\eta_r(e-1)$ of degree $e$ for $\F_{p^{r}}$. 
Then, for any integer $n \geq 1$, we have 
\[
C_{nr} = (-1)^{n-1}C_r^{(n)}
\]
where $C_r^{(n)}$ is the $n$-fold product of $C_r$ with respect to the $(-1)$-composition $\overset{-1}{\ast}$. 
In particular, we have $P_{e,nr}(X)={\rm Char}_X((-1)^{n-1}C_r^{(n)})$. 
\end{thm}

We organize this paper as follows. 
In Section \ref{S2}, 
we review Thaine's results on the $d$-composition of the multiplication 
matrices based on \cite{Tha04}. 
In Section \ref{S3}, 
we study the $d$-compositions of matrices and 
functions in the general situations. 
This enables us to consider the $d$-compositions of 
the multiplication matrices $C_r$ of the Gaussian periods 
and also the $d$-compositions of the Gaussian periods $\eta_r(i)$ 
without linear independence. 
In Section \ref{S4}, 
we recall a Fourier transform on finite abelian groups 
and show that Gaussian periods $\eta_r(i)$ 
and Gauss sums $G_r^*(\chi)$ are each other's finite Fourier transform 
(with some twist). 
Using this, 
we give a proof of Theorem \ref{t1.3}. 
We will give some examples of Theorem \ref{t1.3} in Section \ref{S5}. 
In Section \ref{S6}, the proof of Theorem \ref{t1.4} will be given. 
In Section \ref{S7}, 
we give some examples of Theorem \ref{t1.4} 
for prime degree $e$ with $3\leq e\leq 23$ which also illustrate relations among 
lifts of Jacobi sums, Gaussian periods and multiplication matrices 
of Gaussian periods as in 
Theorem \ref{t1.1}, Theorem \ref{t1.3} and Theorem \ref{t1.4} 
respectively. 

%------------------------------------S2
\section{Thaine's results: compositions of multiplication matrices of Gaussian periods}\label{S2}
We review Thaine's results on the $d$-composition of the multiplication matrices 
based on \cite[Section 1 and Section 2]{Tha04}. 

\subsection{Multiplication matrices of roots}\label{SS2.1}
Thaine \cite[Section 1]{Tha04} defined 
the multiplication matrix $A$ of $\theta_0,\ldots,\theta_{e-1}$ as follows. 

Let $D$ be an integrally closed domain with char $D=0$ 
and $K$ be the quotient field of $D$. 
Let $e \geq 2$ be an integer and 
$P(X)=\sum^{e}_{k=0}c_kX^k = \prod^{e-1}_{i=0}(X-\theta_i) \in D[X]$ be a cyclic polynomial, i.e. an irreducible polynomial with cyclic Galois group over $K$. 
Then $K(\theta_i)/K$ is a cyclic extension of degree $e$ 
with $\Gal(K(\theta_i)/K)=\langle\tau\rangle$. 
We may assume that 
$\tau(\theta_i)=\theta_{i+1}$ where we regard the subscripts modulo $e$. 
We also assume that the discriminant 
$\mathrm{disc}_{K(\theta_i)/K}(\theta_0, \ldots \theta_{e-1})
=(\det [\tau^{j}(\theta_i)]_{0\leq i,j \leq e-1})^2 \neq 0$. 
Then $\{\theta_0,\ldots,\theta_{e-1}\}$ is a vector space basis of 
$K(\theta_i)$. 

Let $M_e(K)$ be the algebra of $e\times e$ matrices over $K$. 
Define the matrix $A=[a_{i,j}]_{0\leq i,j \leq e-1}\in M_e(K)$ 
by $\theta_0\theta_i=\sum^{e-1}_{j=0}a_{i,j}\theta_j$
which is called 
{\it the multiplication matrix of $\theta_0,\ldots,\theta_{e-1}$}. 
Note that the $\theta_i$'s are eigenvalues of $A$ and hence 
$P(X)={\rm Char}_X(A)$; the characteristic polynomial of $A$. 

Thaine \cite[Proposition 1]{Tha04} showed that\\
(i) $a_{i,j}=a_{-i,j-i}$ for $0 \leq i,j \leq e-1$;\\
(ii) $A(\mathcal{K}^{-i}A\mathcal{K}^{i})=(\mathcal{K}^{-i}A\mathcal{K}^{i})A$ 
for $0 \leq i \leq e-1$ 
where $\mathcal{K}=[\delta_{i+1,j}]_{0\leq i,j\leq e-1}$ 
is the $e \times e$ circulant matrix 
\begin{align*}
\mathcal{K} =
\left(
\begin{array}{cccc}
 & 1 & &  \\
 & &  \ddots &  \\
 & &  & 1 \\
1 & & & 
\end{array}\right). 
\end{align*}
Conversely, 
if $A^{\prime}=[a^{\prime}_{i,j}]_{0\leq i,j \leq e-1}\in M_e(K)$ 
satisfies {\rm (i)} and {\rm (ii)} and 
$P^\prime(X)={\rm Char}_X(A^\prime)$ is irreducible over $K$ 
with roots $\theta^{\prime}_{0},\ldots,\theta^{\prime}_{e-1}$, 
then the Galois group of $P^\prime(X)$ over $K$ is cyclic 
and  
$A^{\prime}$ is the multiplication matrix of 
$\theta^{\prime}_{0},\ldots,\theta^{\prime}_{e-1}$ 
with ${\rm disc}_{K(\theta_i)/K}(\theta^{\prime}_{0},\ldots,\theta^{\prime}_{e-1}) \neq 0$.

\subsection{The $d$-compositions of multiplication matrices}\label{SS2.2}
Thaine \cite[Section 2]{Tha04} defined 
the $d$-composition $A\overset{d}{\ast}A^\prime$ 
of the multiplication matrices $A$ and $A^\prime$ as follows. 

Let $P(X)=\prod_{i=0}^{e-1}(X-\theta_i), 
P^{\prime}(X)=\prod_{i=0}^{e-1}(X-\theta^{\prime}_i) \in D[X]$ 
be cyclic polynomials of degree $e$. 
Then $L=K(\theta_i)$ and $L^{\prime}=K(\theta^{\prime}_i)$ 
are cyclic extensions of $K$ of degree $e$ with 
$\Gal(L/K)=\langle\tau\rangle$, 
$\tau(\theta_i)=\theta_{i+1}$ 
and 
$\Gal(L^{\prime}/K)=\langle\tau^{\prime}\rangle$, 
$\tau^{\prime}(\theta^{\prime}_i)=\theta^{\prime}_{i+1}$ 
where we regard the subscripts modulo $e$. 

We assume that $L \cap L^{\prime}= K$. 
We also 
assume that 
$\{\theta_0,\ldots,\theta_{e-1}\}$ 
and 
$\{\theta^{\prime}_0,\ldots,\theta^{\prime}_{e-1}\}$ 
are linearly independent over $K$. 
Then $[L L^{\prime}:K]=e^2$ with 
$G={\rm Gal}(L L^{\prime}/K)\simeq \langle \tau \rangle \times \langle \tau^{\prime} \rangle$ 
and $\{\theta_i\theta^{\prime}_j\mid 0\leq i,j\leq e-1\}$ 
becomes a vector space basis of $L L^{\prime}$ over $K$. 
We make identifications $\tau = (\tau,1)$ 
and $\tau^{\prime} = (1, \tau^{\prime})$.  
For $1\leq d\leq e-1$, 
we have an intermediate field 
$L_d=(L L^{\prime})^{\langle \tau {\tau^{\prime}}^{d} \rangle}=K(\theta_{d,i})$ which is a cyclic extension of $K$ of degree $e$ 
with a normal basis $\{\theta_{d,0},\ldots,\theta_{d,e-1}\}$ 
where $\theta_{d,i} = \sum^{e-1}_{s=0}\theta_s \theta^{\prime}_{ds+i}$. 

Thaine \cite[Proposition 6]{Tha04} (see also \cite[Example 4]{Tha04}) 
proved that the multiplication matrix $A_d$ 
of $\theta_{d,0},\ldots,\theta_{d,e-1}$ is given by 
\begin{align*}
A_d=A\overset{d}{\ast}A^{\prime}
\end{align*}
where $A=[a_{i,j}]_{0\leq i,j\leq e-1}$ and 
$A^{\prime}=[a^{\prime}_{i,j}]_{0\leq i,j\leq e-1}$ are 
the multiplication matrices of 
$\theta_{0},\ldots,\theta_{e-1}$ 
and $\theta^{\prime}_{0},\ldots,\theta^{\prime}_{e-1}$ 
respectively 
and the $d$-composition $A\overset{d}{\ast}A^{\prime}$
of the matrices $A$ and $A^{\prime}$ is defined by 
\[
A \overset{d}{\ast} A^{\prime}:= \left[\sum^{e-1}_{s=0}\sum^{e-1}_{t=0}a_{s,t}a^{\prime}_{ds+i,dt+j}\right]_{0\leq i,j\leq e-1}.
\]

\subsection{Multiplication matrices of Gaussian periods}\label{SS2.3}

The simplest example of the multiplication matrix is that of 
Gaussian periods $\eta_1(0),\ldots,\eta_1(e-1)$ for $\F_{p^1}$ 
(see Section \ref{S1}). 

Let $p$ be a prime 
with $p \equiv 1\ ({\rm mod}\ e)$ and $\eta_1(0),\ldots,\eta_1(e-1)$ 
be the Gaussian periods of degree $e$ for $\F_{p^1}$. 
Then $\eta_1(0),\ldots,\eta_1(e-1)$ are linearly independent over $\Q$ 
and $A=[{\rm Cyc}_1(i,j) -D_{i}f]_{0\leq i,j\leq e-1}$ becomes 
the multiplication matrix of $\eta_1(0),\ldots,\eta_1(e-1)$ 
with $P_{e,1}(X)={\rm Char}_X(A)$ (see the equation (\ref{eq1}) in Section \ref{S1}). 

However, in the general case with $r \geq 2$, 
Gaussian periods $\eta_r(0),\ldots,\eta_r(e-1)$ of degree $e$ for $\F_{p^r}$ are not necessarily 
linearly independent over $\Q$. 
Myerson \cite{Mye81} showed that 
$P_{e,r}(X)$ splits over $\Q$ into $\delta=\gcd(e,(p^r-1)/(p-1))$ factors. 
For example, $P_{e,e}(X)$ splits completely over $\Q$, i.e. 
$\eta_e(0),\ldots,\eta_e(e-1) \in \Q$. 

In the next section, 
we study the $d$-compositions of matrices and 
the $d$-compositions of functions for more general situations. 
This enables us to consider the $d$-compositions of 
the multiplication matrices $C_r$ of the Gaussian periods 
$\eta_r(0),\ldots,\eta_r(e-1)$ 
and also the $d$-compositions of the Gaussian periods $\eta_r(i)$ 
without linear independence. 

%------------------------------------S3
\section{The $d$-compositions}\label{S3}
\subsection{The $d$-compositions of matrices}\label{SS3.1}

Let $K$ be a field with char $K=0$. 
Let $e\geq 2$ be an integer and 
$M_e(K)$ be the algebra of $e\times e$ matrices over $K$. 

According to Thaine \cite{Tha04} (see Subsection \ref{SS2.2}), 
for $A=[a_{i,j}]_{0\leq i,j\leq e-1}$, 
$B=[b_{i,j}]_{0\leq i,j\leq e-1} \in M_e(K)$ 
and $d\in (\Z/e\Z)\setminus\{0\}$, 
we define {\it the $d$-composition $A \overset{d}{\ast} B$ of $A$ and $B$} as 
\[
A \overset{d}{\ast} B:= \left[\sum^{e-1}_{s=0}\sum^{e-1}_{t=0}a_{s,t}b_{ds+i,dt+j}\right]_{0\leq i,j\leq e-1}.
\]
We see that the $d$-composition $\overset{d}{\ast}$ and 
the ordinary addition $+$ 
satisfy the distributive law, i.e. 
\begin{align*}
A \overset{d}{\ast} (B + C) = (A \overset{d}{\ast}B) + (A \overset{d}{\ast}C),\ 
(A + B) \overset{d}{\ast} C = (A \overset{d}{\ast}C) + (B \overset{d}{\ast}C), 
\end{align*}
although 
it does not satisfy the cancellation law, i.e. there exist matrices $A, B, C$ 
such that $A \overset{d}{\ast} C = B \overset{d}{\ast} C$ and $A \neq B$. 

\begin{prop}\label{p3.1}
Let $A=[a_{i,j}]_{0\leq i,j\leq e-1}$, 
$B=[b_{i,j}]_{0\leq i,j\leq e-1} \in M_e(K)$. 
We write $A[i,j]=a_{i,j}$ for convenience. 
For $d \in (\Z/e\Z)^{\times}$, we have
\[
(B \overset{d}{\ast} A)[i,j] = (A \overset{d^{-1}}{\ast} B)[-d^{-1}i,-d^{-1}j].
\]
In particular, we get
\[
A \overset{-1}{\ast} B = B \overset{-1}{\ast} A.
\]
\end{prop}
\begin{proof}
Putting $s^{\prime}:=ds+i$ and $t^{\prime}:=dt+j$, we have
\begin{align*}
(B \overset{d}{\ast} A)[i,j]
&= \sum^{e-1}_{s=0}\sum^{e-1}_{t=0}B[s,t]A[ds+i,dt+j]\\
&= \sum^{e-1}_{s^{\prime}=0}\sum^{e-1}_{t^{\prime}=0}
B[d^{-1}(s^{\prime}-i),d^{-1}(t^{\prime}-j)]A[s^{\prime},t^{\prime}]\\
&= \sum^{e-1}_{s^{\prime}=0}\sum^{e-1}_{t^{\prime}=0}A[s^{\prime},t^{\prime}]B[d^{-1}s^{\prime}+(-d^{-1}i),d^{-1}t^{\prime}+(-d^{-1}j)]\\
&=(A \overset{d^{-1}}{\ast} B)[-d^{-1}i,-d^{-1}j].
\end{align*}
\end{proof}
\begin{cor}\label{c3.2}
We have 
${\rm Char}_X(B \overset{d}{\ast} A)={\rm Char}_X(A \overset{d^{-1}}{\ast} B)$ 
where ${\rm Char}_X(A)$ stands for the characteristic polynomial of $A$.
\end{cor}
\begin{proof}
By Proposition \ref{p3.1}, we see that there exists 
an invertible matrix $P \in M_e(K)$ such that
\[
B \overset{d}{\ast}A =P^{-1}(A \overset{d^{-1}}{\ast} B)P.
\]
Hence the assertion follows. 
\end{proof}

\begin{lem}\label{l3.3}
Let $A=[a_{i,j}]_{0\leq i,j\leq e-1}$, 
$B=[b_{i,j}]_{0\leq i,j\leq e-1} \in M_e(K)$. 
Let $\mathcal{K}=[\delta_{i+1,j}]_{0\leq i,j\leq e-1} \in M_e(K)$ 
be a circulant matrix. 
For $0\leq l \leq e-1$, we have\\
{\rm (i)} $\mathcal{K}^{l}(A \overset{d}{\ast} B) = A \overset{d}{\ast} (\mathcal{K}^{l}B) = (\mathcal{K}^{-d^{-1}l}A) \overset{d}{\ast} B$,\\
{\rm (ii)} $(A \overset{d}{\ast} B)\mathcal{K}^{l} = A \overset{d}{\ast} (B\mathcal{K}^{l}) = (A\mathcal{K}^{-d^{-1}l}) \overset{d}{\ast} B$.
\end{lem}
\begin{proof}
Write $A[i,j]=a_{i,j}$. 
We first see that
\[
(\mathcal{K}^{l}A)[i,j]=A[i+l,j],\ (A\mathcal{K}^{l})[i,j]=A[i,j-l].
\]
(i) The first equality follows from 
\begin{align*}
(\mathcal{K}^{l}(A \overset{d}{\ast} B))[i,j] &= (A \overset{d}{\ast} B)[i+l,j]\\
&= \sum^{e-1}_{s=0}\sum^{e-1}_{t=0}A[s,t]B[ds+(i+l),dt+j]\\
&=\sum^{e-1}_{s=0}\sum^{e-1}_{t=0}A[s,t](\mathcal{K}^{l}B)[ds+i,dt+j]\\
&=(A \overset{d}{\ast}(\mathcal{K}^{l}B))[i,j].
\end{align*}
By substituting $s^{\prime}:= s+d^{-1}l$, 
the second equality follows from 
\begin{align*}
(A \overset{d}{\ast} (\mathcal{K}^{l}B))[i,j] &= \sum^{e-1}_{s=0}\sum^{e-1}_{t=0}A[s,t]\mathcal{K}^{l}B[ds+i,dt+j]\\
&=\sum^{e-1}_{s^{\prime}=0}\sum^{e-1}_{t=0}A[s^{\prime}-d^{-1}l,t]B[ds^{\prime}+i,dt+j] \\
&=\sum^{e-1}_{s^{\prime}=0}\sum^{e-1}_{t=0}(\mathcal{K}^{-d^{-1}l})A[s^{\prime},t]B[ds^{\prime}+i,dt+j]\\
&=((\mathcal{K}^{-d^{-1}l}A) \overset{d}{\ast} B)[i,j].
\end{align*}
(ii) can be proved in the similar way, and we omit the proof.
\end{proof}

\begin{prop}\label{p3.4}
For $A, B, C \in M_e(K)$ and $d_1, d_2 \in (\Z/e\Z)^{\times}$, we have
\[
A \overset{d_1}{\ast} (B \overset{d_2}{\ast} C) = (A \overset{-d_2^{-1}d_1}{\ast} B) \overset{d_2}{\ast} C.
\]
In particular, we get 
\[
 A \overset{d_1}{\ast} (B \overset{-1}{\ast} C) = (A \overset{d_1}{\ast} B) \overset{-1}{\ast} C
\] 
and hence the $(-1)$-composition 
$\overset{-1}{\ast}$ satisfies the associative law. 
\end{prop}

\begin{proof}
For $A=[a_{i,j}]_{0\leq i,j\leq e-1}\in M_e(K)$, 
we write $A[i,j]=a_{i,j}$. 
We obtain the following expression of 
$A \overset{d}{\ast} B$ by using the circulant matrix 
$\mathcal{K}=[\delta_{i+1,j}]_{0\leq i,j\leq e-1} \in M_e(K)$: 
\[
A \overset{d}{\ast} B = \sum^{e-1}_{s=0}\sum^{e-1}_{t=0}A[s,t]\,\mathcal{K}^{ds}B\mathcal{K}^{-dt}.
\]
Then it follows from Lemma \ref{l3.3} that 
\begin{align*}
A\overset{d_1}{\ast}(B \overset{d_2}{\ast} C)
&=-\sum^{e-1}_{s=0}\sum^{e-1}_{t=0}A[s,t]\,
\mathcal{K}^{d_1s}(B\overset{d_2}{\ast}C)\mathcal{K}^{-d_1t}\\
&=\sum^{e-1}_{s=0}\sum^{e-1}_{t=0} A[s,t]\left((\mathcal{K}^{-d_2^{-1}d_1s}B\mathcal{K}^{d_2^{-1}d_1t})\overset{d_2}{\ast}C\right)\\
&=\sum^{e-1}_{s=0}\sum^{e-1}_{t=0} A[s,t]
\left(  \sum^{e-1}_{u=0}\sum^{e-1}_{v=0}\mathcal{K}^{-d_2^{-1}d_1s}B\mathcal{K}^{d_2^{-1}d_1t}[u,v]\,\mathcal{K}^{d_2u}C\mathcal{K}^{d_2v}\right)\\
&=\sum^{e-1}_{u=0}\sum^{e-1}_{v=0} \left(\sum^{e-1}_{s=0}\sum^{e-1}_{t=0} A[s,t] \mathcal{K}^{-d_2^{-1}d_1s}B\mathcal{K}^{d_2^{-1}d_1t}\right)\![u,v]\,\mathcal{K}^{d_2u}C\mathcal{K}^{d_2v}\\
&=\sum^{e-1}_{u=0}\sum^{e-1}_{v=0}(A\overset{-d_2^{-1}d_1}{\ast}B)[u,v]\,\mathcal{K}^{d_2u}C\mathcal{K}^{d_2v}\\
&= (A \overset{-d_2^{-1}d_1}{\ast} B) \overset{d_2}{\ast} C.
\end{align*}
The last assertion follows if we take $d_2=-1$. 
\end{proof}

By the last statement of Proposition \ref{p3.4}, 
we can define the $n$-fold product $A^{(n)}$ as 
\begin{align*}
A^{(n)} := A \overset{-1}{\ast} \cdots \overset{-1}{\ast} A\quad(\text{$n$-fold}).%\label{defAn}
\end{align*}

%%%%%%%%%%%%%%%%%%%%%%%%%%%%%%%%%%%
\subsection{The $d$-compositions of functions}\label{SS3.2}
Let $e \geq 2$ be an integer and $P(X)=\sum^{e-1}_{k=0}c_kX^k = \prod^{e-1}_{i=0}(X-\theta_i) \in D[X]$ be a cyclic polynomial. 
Let $L^2(\Z/e\Z)$ be the vector space of all $\C$-valued functions 
on $\Z/e\Z$. 
We may regard $\theta_i=\theta(i)$ as 
the function $\theta:\Z/e\Z\to \mathbb{C}$, i.e. $\theta\in L^2(\Z/e\Z)$. 
Based on the results in 
Subsection \ref{SS3.1}, 
for $f,g \in L^2(\Z/e\Z)$ and $d\in (\Z/e\Z)\setminus\{0\}$, 
we define {\it the $d$-composition $f \overset{d}{\ast} g$ of $f$ and $g$} by
\[
(f \overset{d}{\ast} g)(i) := \sum^{e-1}_{s=0}f(s)g(ds+i).
\]
In particular, we get
\[
(f \overset{-1}{\ast} g)(i) = \sum_{k_1+k_2\equiv\,i\,({\rm mod} e)\atop 0\leq k_1,k_2\leq e-1}f(k_1)g(k_2)=(f*g)(i)
\]
where $*$ is the (usual) convolution 
of $L^2(\Z/e\Z)$ (see also Section \ref{S4}). 
Hence the $(-1)$-composition $\overset{-1}{\ast}$ 
satisfies the commutative and the associative laws, 
and we can define the $n$-fold product of $f$ 
with respect to $\overset{-1}{\ast}$ as 
\[
f^{\left( n \right)}(i):= (f\overset{-1}{\ast} \cdots \overset{-1}{\ast}f)(i) \quad(\text{$n$-fold}).
\]
By the definition, we get 
\[
f^{(n)}(i)=\sum_{k_1+\cdots+k_n\equiv\,i\,({\rm mod} e)\atop 0\leq k_1,\ldots,k_n\leq e-1}f(k_1)\cdots f(k_n).
\]

In order to prove Theorem \ref{t1.4}, 
we need the following proposition which gives the relation 
between $A \overset{d}{\ast} B$ and $f \overset{d}{\ast}g$:
\begin{prop}\label{p3.5}
Let $K$ be a field with ${\rm char}$ $K=0$. 
For $f, g\in L^2(\Z/e\Z)$, 
we assume that there exist 
$A=[a_{i,j}]_{0\leq i,j \leq e-1}$, 
$B=[b_{i,j}]_{0\leq i,j \leq e-1}\in M_e(K)$ such that 
\[
f(i)\,f(j)=\sum^{e-1}_{k=0}a_{j-i,k-i}\,f(k)\ {\rm and}\ 
g(i)\,g(j)=\sum^{e-1}_{k=0}b_{j-i,k-i}\,g(k).
\]
Then we have 
\[
(f \overset{d}{\ast} g)(i)\,(f \overset{d}{\ast} g)(j)
=\sum^{e-1}_{k=0}(A \overset{d}{\ast} B)[j-i,k-i]\,(f \overset{d}{\ast} g)(k)
\]
where we write $A[i,j]=a_{i,j}$. 
In particular, 
\[
f^{(n)}(i)\,f^{(n)}(j)=\sum^{e-1}_{j=0} A^{(n)}[j-i,k-i]\,f^{(n)}(j)
\]
where $A^{(n)}$ is the $n$-fold product of $A$ 
with respect to the $(-1)$-composition $\overset{-1}{\ast}$. 
\end{prop}
\begin{proof}
By substituting $k:=v-du$, $s:=n-m$, $t:=u-m$, we have
\begin{align*}
&(f \overset{d}{\ast} g)(i)(f \overset{d}{\ast} g)(j)\\
&=\left( \sum_{m=0}^{e-1}f(m)g(dm+i) \right)\left( \sum_{n=0}^{e-1}f(n)g(dn+j)\right) \\
&=\sum_{m,n=0}^{e-1}(f(m)f(n))(g(dm+i)g(dn+j)) \\
&=\sum_{m,n=0}^{e-1}\left( \sum_{u=0}^{e-1} A[n-m,u-m]f(u)\right) \left( \sum_{v=0}^{e-1} B[d(n-m)+j-i,v-(dm+i)]g(v)\right)\\
&=\sum_{m,n,u,k=0}^{e-1}A[n-m,u-m]B[d(n-m)+j-i,d(u-n)+k-i]f(u)g(du+k)\\
&=\sum_{s,t,u,k=0}^{e-1}A[s,t]B[ds+j-i,dt+k-i]f(u)g(du+k) \\
&= \sum_{k=0}^{e-1}\left( \sum_{s,t=0}^{e-1}A[s,t]B[ds+(j-i),dt+(k-i)]\right)\left( \sum_{u=0}^{e-1} f(u)g(du+k) \right) \\
&=\sum_{k=0}^{e-1}(A\overset{d}{\ast}B)[j-i,k-i](f\overset{d}{\ast}g)(k). 
\end{align*}
\end{proof}

%%%%%%%%%%%%%%%%%%%%%%%%%%%%%%
Applying Proposition \ref{p3.5} 
for Gaussian periods $\eta_r(i)$ of degree $e$ for $\F_{p^r}$ 
and $\eta^\prime_s(i)$ for $\F_{q^s}$ 
with $p^r, q^s\equiv 1\ ({\rm mod}\ e)$ 
(we may apply the both cases $p\neq q$ and $p=q$)
and their multiplication matrices, we get: 

\begin{cor}
Let $e \geq 2$ be an integer and 
$p^{r}$ $($resp. $q^{s}$$)$ be a prime power with 
$p^r\equiv 1\ ({\rm mod}\ e)$ $($resp. $q^s\equiv 1\ ({\rm mod}\ e)$$)$. 
$($We may take $p^r$, $q^s$ in the both cases $p\neq q$ and $p=q$$.)$ 
We regard the Gaussian periods $\eta_r(i)$ of degree $e$ 
for $\F_{p^{r}}$ as the functions from $\Z/e\Z$ to $\C$, 
i.e. $\eta_r \in L^2(\Z/e\Z)$. 
Let $C$ $($resp. $C^\prime$$)$ 
be the multiplication matrix of the Gaussian periods 
$\eta_r(0),\ldots,\eta_r(e-1)$ of degree $e$ 
for $\F_{p^r}$ 
$($resp. 
$\eta^\prime_s(0),\ldots,\eta^\prime_s(e-1)$ of degree $e$ 
for $\F_{q^s}$$)$. 
Then we have
\[
(\eta_{r}\overset{d}{\ast} \eta^\prime_{s})(i)\, 
(\eta_{r}\overset{d}{\ast} \eta^\prime_{s})(j)
=\sum^{e-1}_{k=0}(C\overset{d}{\ast}C^\prime)[j-i,k-i]\, 
(\eta_{r}\overset{d}{\ast}\eta^\prime_{s})(k).
\]
\end{cor}

For the case $p\neq q$, we can find examples of 
$C\overset{d}{\ast}C^\prime$ in Thaine \cite[Example 4, page 259]{Tha04}. 
We will treat the case $p=q$ in the remaining part of this paper. 

%------------------------------------S5
\section{Proof of Theorem \ref{t1.3}}\label{S4} 
We recall a Fourier transform on finite abelian groups 
(see Terras \cite[Chapter 10]{Ter99}). 
Let $G$ be a finite abelian group and 
\[
L^2(G) = \{f : G \to \C\}
\]
be the vector space of all $\C$-valued functions on $G$ with 
the inner product 
$\langle f,g\rangle=\sum_{x\in G}f(x)\overline{g(x)}$. 
Let $\widehat{G}={\rm Hom}(G,\C^{\times})$ be the dual 
of $G$. 
For $f\in L^2(G)$, 
{\it the finite Fourier transform $\mathscr{F}(f)=\widehat{f}\in L^2(\widehat{G})$ of $f$} 
is defined to be 
\begin{align*}
(\mathscr{F}(f))(\chi) = \widehat{f}(\chi) =\sum_{x \in G}f(x)\overline{\chi(x)}=\langle f,\chi\rangle. 
\end{align*} 
Then $\mathscr{F}:L^2(G)\rightarrow L^2(\widehat{G})$ becomes a 
bijective linear transformation with the inverse 
\begin{align*}
(\mathscr{F}^{-1}(\widehat{f}))(x) = f(x) 
= \frac{1}{\# G}\sum_{\chi \in \widehat{G}}\widehat{f}(\chi)\chi(x)
=\frac{1}{\# G}\sum_{\chi \in \widehat{G}}\langle f,\chi\rangle\chi(x).
\end{align*}
For $f, g\in L^2(G)$, we define {\it the convolution $f \ast g\in L^2(G)$ 
of $f$ and $g$} by
\[
(f \ast g)(x) =\sum_{y \in G} f(y)g(x-y).
\]
Then the space $L^2(G)$ with the convolution $\ast$
is isomorphic to the group ring $\C[G]$ (with the usual convolution product) 
as a commutative $\C$-algebra by 
$L^2(G)\ni f\mapsto \sum_{x\in G}f(x)x\in \C[G]$. 
We also have the compatibility of the convolution $\ast$ and 
the finite Fourier transform $\mathscr{F}(f)=\widehat{f}$: 
\begin{align}
\widehat{(f\ast g)}(\chi)=\widehat{f}(\chi)\widehat{g}(\chi)\label{eq4}
\end{align}
(see Terras \cite[page 168, Theorem 2]{Ter99}). 

In order to show Theorem \ref{t1.3}, 
we prepare the following fundamental lemma: 
\begin{lem}\label{l4.1}
Let $e \geq 2$ be an 
integer and $p^r$ be a prime power 
with $p^r\equiv 1\ ({\rm mod}\ e)$. 
Let $\gamma$ be a fixed generator of $\F_{p^r}^{\times}$ and $\chi$ be the character on $\F_{p^r}$ 
with $\chi(\gamma)=\zeta_e$ and $\chi(0)=0$. 
We regard the Gaussian periods $\eta_r(i)$ of degree $e$ 
for $\F_{p^{r}}$ as the functions from $\Z/e\Z$ to $\C$, 
i.e. $\eta_r \in L^2(\Z/e\Z)$, and 
the Gauss sum $G_r^{\ast}(\chi^j)$ for $\F_{p^{r}}$
as the functions from 
$\widehat{\Z/e\Z}$ to $\C$, i.e. $G_r^{\ast}\in L^2(\widehat{\Z/e\Z})$. 
Then the finite Fourier transform $\mathscr{F}(\eta_r)$ of $\eta_r$ 
is given by 
\begin{align*}
(\mathscr{F}(\eta_r))(\chi^j)=G_{r}^{\ast}(\chi^{-j})
\end{align*}
and we also have
\begin{align*}
(\mathscr{F}^{-1}(G_{r}^{\ast}))(i)=\eta_{r}(-i).
\end{align*}
\end{lem}
\begin{proof}
Because $\F_{p^r}^\times$ is a cyclic group of order $p^r-1=ef$
and $\chi\in\widehat{\F_{p^r}^\times}$ is of order $e$,
the Gauss sum $G_r^\ast(\chi^j)$
can be regarded as
the function from $\widehat{\Z/e\Z}$ to $\C$ via factors through
$G_r^\ast:\widehat{\F_{p^r}^\times}\rightarrow
\widehat{\F_{p^r}^\times}/H\rightarrow\C$
where $H$ is the group of $e$-th powers of $\widehat{\F_{p^r}^{\times}}$
with $\widehat{\F_{p^r}^\times}/H\simeq \widehat{\Z/e\Z}$.
Note that $\{\chi^j H\mid 0\leq j\leq e-1\}$ gives
a set of complete representatives for
$\widehat{\F_{p^r}^\times}/H\simeq \widehat{\Z/e\Z}$.

We have the following 
well-known relations between the Gauss sums $G_r^\ast(\chi^j)$ 
and the Gaussian periods $\eta_r(i)$: 
\begin{align*}%\label{eq2}
G_{r}^{\ast}(\chi^j)=\sum_{i=0}^{e-1}\zeta_e^{ij}\eta_r(i),\quad 
\eta_r(i)=\frac{1}{e}\sum_{j=0}^{e-1}\zeta_e^{-ij}G_{r}^{\ast}(\chi^j)
\end{align*}
(see \cite[Proposition 1 (f)]{Mye81}). 
Then it follows that 
\begin{align*}
(\mathscr{F}(\eta_r))(\chi^j)
=\sum_{x \in \Z/e\Z}\eta_r(x)\overline{\chi^j}(x)
=\sum_{i=0}^{e-1}\zeta_e^{-ij}\eta_r(i)
=G_{r}^{\ast}(\chi^{-j})
\end{align*}
and 
\begin{align*}
(\mathscr{F}^{-1}(G_{r}^{\ast}))(i)
&=\frac{1}{\#(\Z/e\Z)}\sum_{\psi \in \widehat{\Z/e\Z}}G_{r}^{\ast}(\psi)\psi(i)\\
&=\frac{1}{e}\sum_{j=0}^{e-1} G_{r}^{\ast}(\chi^j)\chi^j(i)
=\frac{1}{e}\sum_{j=0}^{e-1} \zeta_e^{ij} G_{r}^{\ast}(\chi^j)
=\eta_{r}(-i).
\end{align*}
\end{proof}

{\it Proof of Theorem \ref{t1.3}}. 
By Theorem \ref{t1.1}, 
the equation (\ref{eq4}) and Lemma \ref{l4.1}, we have 
\begin{align*}
\widehat{\eta_r^{(n)}}(\chi^i)=(\widehat{\eta_r}(\chi^i))^n
=(G_r^*(\chi^{-i}))^n=(-1)^{n-1}G^*_{nr}(\chi^{-i})
=(-1)^{n-1}\widehat{\eta_{nr}}(\chi^i).
\end{align*}
We get 
$\widehat{\eta_r^{(n)}}=(-1)^{n-1}\widehat{\eta_{nr}}$ 
and hence 
$\eta_r^{(n)}=(-1)^{n-1}\eta_{nr}$.\qed

%%%%%%%%%%%%%%%%%%%%%%%%%%%%%%%%%%%%%%%%%%%%%%%%%%%%%%%%%%
\section{Examples of Theorem \ref{t1.3}}\label{S5} 
{\it Exponential Gauss sums $g_r(b,e)$ $(b \in \F_{p^r})$ 
of degree $e$ for $\F_{p^r}$} are defined by
\[
g_r(b,e):=\sum_{\alpha \in \F_{p^r}}\zeta_p^{{\rm Tr}(b\alpha^e)}.
\]
We write $g_r(e):=g_r(1,e)$. 
We also define {\it the reduced Gaussian periods} 
$\eta_r^*(i)$ of order $e$ for $\F_{p^r}$ as 
\[
\eta_r^*(i):=e\eta_r(i)+1=g_r(\gamma^i,e)
\]
for $0\leq i\leq e-1$ (see \cite[page 327]{BEW98}). 
Let $\chi$ be the character on $\F_{p^r}$ with 
$\chi(\gamma)=\zeta_e$ and $\chi(0)=0$ where 
$\langle\gamma\rangle=\F_{p^r}^\times$ as before. 
We see that the Gauss sums $G^*_r(\chi^j)$ 
and the reduced Gaussian periods $\eta_r^*(i)$ 
satisfy the following relations 
\begin{align*}
G^*_r(\chi^j)&=
\begin{cases}
\displaystyle{\frac{1}{e}\sum_{i=0}^{e-1}\zeta_e^{ij}\eta^*_r(i)}&{\rm if}\quad 1\leq j\leq e-1\\
-1&{\rm if}\quad j=0,
\end{cases}\\
\eta_r^{\ast}(i)&=\sum_{j=1}^{e-1}\zeta_e^{-ij}G^*_{r}(\chi^j)
\end{align*}
(see \cite[Proposition 1 (g)]{Mye81}, 
\cite[page 332, Theorem 10.10.8]{BEW98}, 
cf. the proof of Lemma \ref{l4.1}).
In particular, we have
\begin{align*}
g_r(e)=\eta^*_r(0)=\sum_{j=1}^{e-1}G^*_r(\chi^j)=\sum_{j=1}^{e-1}G_r(\chi^j).
\end{align*}

We also use {\it the reduced period polynomial} 
$P_{e,r}^{\ast}(X):=\prod^{e-1}_{i=0}\left(X-\eta^{\ast}_r(i)\right)$ 
of degree $e$ for $\F_{p^r}$ 
with the coefficient of $X^{e-1}$ zero. 
An explicit determination of the factors of $P_{e,r}^*(X)$ 
is important because the exponential Gauss sum $g_r(e)=\eta^*_r(0)$ 
becomes a root of $P_{e,r}^*(X)$ (see also Section \ref{S7}). 

%For further applications to the coding theory, 
%see e.g. McEliece and Rumsey \cite{MR72} 
%and \cite[Section 11.7]{BEW98}.\\ 

By applying Theorem \ref{t1.3}, we can obtain the reduced 
Gaussian periods 
$\eta_e^*(i)$ $(i=0,\ldots,e-1)$ 
of order $e$ for $\F_{p^e}$ as follows (we take the generator $\gamma$ of $\F_p^\times=\langle \gamma \rangle$ as the smallest one):

(1) $e=3$. 
We take $p=7=ef+1$ with $f=2$ and $\gamma=3$. 
Then $\eta_3^*(i)$ $(i=0,1,2)$ are given by 
$7$, $-35$, $28$ (in this order).
%[ 7, -35, 28 ]

(2) $e=5$. 
We take $p=11=ef+1$ with $f=2$ and $\gamma=2$. 
Then 
$\eta_5^*(i)$ $(i=0,\ldots,4)$ are given by 
$-979$, $-649$, $1276$, $-99$, $451$. 
%[ -979, -649, 1276, -99, 451 ]

(3) $e=7$. 
We take $p=29=ef+1$ with $f=4$ and $\gamma=2$. 
Then 
$\eta_7^*(i)$ $(i=0,\ldots,6)$ are given by 
$-317869$, $-259405$, $-324771$, $442569$, $233682$, $-182671$, $408465$. 
%[ -317869, -259405, -324771, 442569, 233682, -182671, 408465 ]
%[ 29, 97, 113 ]

(4) $e=11$. 
We take $p=23=ef+1$ with $f=2$ and $\gamma=5$. 
Then 
$\eta_{11}^*(i)$ $(i=0,\ldots,10)$ are given by 
$52918009$, $3199967$, $-202694722$, $-64390754$, $142959444$, 
$-23093817$, $166665038$, $-19592803$, $47121273$, $-58652208$, $-44439427$. 
%[ 52918009, 3199967, -202694722, -64390754, 142959444, -23093817, 166665038, -19592803, 47121273, -58652208, -44439427 ]
%[ 23, 53, 43411 ]

%%%%%%%%%%%%%%%%%%%%%%%%%%%%%%%%%%%%%%%%%%%%%%%%%%%%%%%%%%%%%%%%%%%
(5) $e=13$. 
%Let $p$ be a prime with $p\equiv 1\ ({\rm mod}\ 13)$. 
We take $p=53=ef+1$ with $f=4$ and $\gamma=2$. 
Then 
$\eta_{13}^*(i)$ $(i=0,\ldots,12)$ are given by 
$782475795674$, $338244988654$, $-245670171356$, $83828569254$, 
$-740552966334$, \\
$-910543059425$, $117899008800$, $664438112586$, $-186980700750$, $-238169301889$, $-277653262665$, $1040615291340$, $-427932303889$. 
%[ 782475795674, 338244988654, -245670171356, 83828569254, -740552966334, -910543059425, 117899008800, 664438112586, -186980700750, -238169301889, -277653262665, 1040615291340, -427932303889 ]
%[ 2, 53, 7381847129 ]

%%%%%%%%%%%%%%%%%%%%%%%%%%%%%%%%%%%%%%%%%%%%%%%%%%%%%%%%%%%%%%%%%%%
(6) $e=17$. 
We take $p=103=ef+1$ with $f=6$ and $\gamma=5$. 
Then 
$\eta_{17}^*(i)$ $(i=0,\ldots,16)$ are given by 
$-651513206543247755$, 
$670088231006862759$, 
$-373934090375919493$,\\
$587253242462231659$, 
$-243310155546790559$,
$163898849457734107$, 
$-197783211402587952$, \\ 
$-1253189038565026183$, 
$35922811461007315$, 
$356621718684896633$, 
$-478731856802195967$, \\
$-289516205265127375$, 
$461908111585063663$,
$464742031061114921$, 
$670357206530506901$, \\
$282238003107978403$, 
$-205052440856501077$.\\ 
%[ -651513206543247755, 670088231006862759, -373934090375919493, 587253242462231659, -243310155546790559, 163898849457734107, -197783211402587952, -1253189038565026183, 35922811461007315, 356621718684896633, -478731856802195967, -289516205265127375, 461908111585063663, 464742031061114921, 670357206530506901, 282238003107978403, -205052440856501077 ]
%[ 5, 7, 103, 172709, 1046412659 ]

%%%%%%%%%%%%%%%%%%%%%%%%%%%%%%%%%%%%%%%%%%%%%%%%%%%%%%%%%%%%%%%%%%%%
%We give GAP (\cite{GAP}) computations for 
%(6) $e=17$, $p=103$ and $\gamma=5$ in the arXiv version of this paper 
%\cite[Section \ref{S5}]{HK}.
%%%%%%%%%%%%%%%%%%%%%%%%%%%%%%%%%%%%%%%%%%%%%%%%%%%%%%%%%%%%%%%%%%%%
We give GAP (\cite{GAP}) computations for 
(6) $e=17$, $p=103$ and $\gamma=5$. 
The cases (1)--(5) can be obtained by the similar manner. 

\begin{verbatim}
gap> etaf:=function(e,p,i)
> local g;
> g:=PrimitiveRootMod(p);
> if i<0 then i:=i mod e;
> fi;
> return Sum([1..(p-1)/e],j->E(p)^(g^(e*j+i)));
> end;
function( e, p, i ) ... end
gap> e:=17;;p:=103;;PrimitiveRootMod(p); # e=17, p=103, g=5
5
gap> eta:=function(i)
> return(etaf(e,p,i));
> end;
function( i ) ... end
gap> p2:=function(i)
> return Sum([0..e-1],k1->eta(k1)*eta(-k1+i));
> end;
function( i ) ... end
gap> p4:=function(i)
> return Sum([0..e-1],k1->p2(k1)*p2(-k1+i));
> end;
function( i ) ... end
gap> p8:=function(i)
> return Sum([0..e-1],k1->p4(k1)*p4(-k1+i));
> end;
function( i ) ... end
gap> p9:=function(i)
> return Sum([0..e-1],k1->p8(k1)*eta(-k1+i));
> end;
function( i ) ... end
gap> p17:=function(i)
> return Sum([0..e-1],k1->p9(k1)*p8(-k1+i));
> end;
function( i ) ... end
gap> L:=List([0..e-1],i->p17(i)); # Gaussian periods eta_{e,i}: i=0,...16
[ -38324306267249868, 39416954765109574, -21996122963289382, 
  34544308380131274, -14312362090987680, 9641108791631418, 
  -11634306553093409, -73717002268530952, 2113106556529842, 
  20977748157935096, -28160697458952704, -17030365015595728, 
  27171065387356686, 27337766533006760, 39432776854735700, 
  16602235476939906, -12061908285676534 ]
gap> 17*L+1; # reduced Gaussian periods eta^*_{e,i}: i=0,...16
[ -651513206543247755, 670088231006862759, -373934090375919493, 
  587253242462231659, -243310155546790559, 163898849457734107, 
  -197783211402587952, -1253189038565026183, 35922811461007315, 
  356621718684896633, -478731856802195967, -289516205265127375, 
  461908111585063663, 464742031061114921, 670357206530506901, 
  282238003107978403, -205052440856501077 ]
\end{verbatim}

%%%%%%%%%%%%%%%%%%%%%%%%%%%%%%%%%%%%%%%%%%%%%%%%%%%%%%%%%%%%%%%%%%%%
\section{Proof of Theorem \ref{t1.4}}\label{S6} 
For $A=[a_{i,j}]_{0\leq i,j\leq e-1}\in M_e(K)$, 
we write $A[i,j]=a_{i,j}$ for convenience. 
By the definition, for the multiplication matrix $C_{nr}$ of 
Gaussian periods $\eta_{nr}(0),\ldots,\eta_{nr}(e-1)$, we have
\[
\eta_{nr}(i)\eta_{nr}(j)=
\sum^{e-1}_{k=0}C_{nr}[j-i,k-i]\eta_{nr}(k).
\]
By Theorem \ref{t1.3} (the dual form of Davenport and Hasse's lifting theorem), we obtain 
\[
\eta_{r}^{(n)}(i)\eta_{r}^{(n)}(j)=
\sum^{e-1}_{k=0}(-1)^{n-1}C_{nr}[j-i,k-i]\eta_{r}^{(n)}(k).
\]
By setting $l:=j-i$ and $m:=k-i$, we get
\begin{align}\label{eq5}
\eta_{r}^{(n)}(i)\eta_{r}^{(n)}(i+l)=
\sum^{e-1}_{m=0}(-1)^{n-1}C_{nr}[l,m]\eta_{r}^{(n)}(i+m).
\end{align}
On the other hand, by Proposition \ref{p3.5}, we have
\[
\eta_{r}^{(n)}(i)\eta_{r}^{(n)}(j)=
\sum^{e-1}_{k=0}C_r^{(n)}[j-i,k-i]\eta_{r}^{(n)}(k).
\]
By setting $l:=j-i$ and $m:=k-i$ also, we obtain
\begin{align}\label{eq6}
\eta_{r}^{(n)}(i)\eta_{r}^{(n)}(i+l)=
\sum^{e-1}_{m=0}C_r^{(n)}[l,m]\eta_{r}^{(n)}(i+m).
\end{align}
Equations (\ref{eq5}) and (\ref{eq6}) imply that $(-1)^{n-1}C_{nr}$ and $C_r^{(n)}$ have the same eigenvalues $\eta_{r}^{(n)}(i)$ $(0\leq i \leq e-1)$ 
and the same eigenvector 
\[
T_i=(\eta_{r}^{(n)}(i),\eta_{r}^{(n)}(i+1),\ldots,\eta_{r}^{(n)}(0),\eta_{r}^{(n)}(1),\ldots,\eta_{r}^{(n)}(i-1))^t
\]
with respect to 
$\eta_{r}^{(n)}(i)$ where $t$ stands for the transposed vector. 

We take the circulant matrix $P:=(T_0,\ldots,T_{e-1})$ with 
determinant $\prod_{j=0}^{e-1}f(\zeta_e^j)$ 
where $f(x)=\sum_{i=0}^{e-1} \eta_{r}^{(n)}(i)x^i$. 
We see that the matrix $P$ is invertible because 
there exist at least two distinct $\eta_{r}^{(n)}(i)=(-1)^{n-1}\eta_{nr}(i)$
by Baumert, Mills and Ward \cite[Lemma 2 and the proof of Lemma 3]{BMW82}. 
Hence the both 
$(-1)^{n-1}C_{nr}$ and $C_r^{(n)}$ are diagonalized by the same $P$: 
\[
P^{-1}(-1)^{n-1}C_{nr}P=
\begin{pmatrix}
 \eta_{r}^{(n)}(0) & &  \\
 & \ddots &  \\
 & & \eta_{r}^{(n)}(e-1) \\ 
\end{pmatrix}
=P^{-1}C_r^{(n)}P.
\]
This implies that $(-1)^{n-1}C_{nr}=C_r^{(n)}$.
\qed

%%%%%%%%%%%%%%%%%%%%%%%%%%%%%%%%%%%%%%%%%%%%%%%%%%%%%%%
\section{Applications of Theorem \ref{t1.4}}\label{S7} 

We give some applications of Theorem \ref{t1.4} 
when $e=l$ is an odd prime 
which illustrate relations among 
lifts of Jacobi sums, Gaussian periods and multiplication matrices 
of Gaussian periods as in 
Theorem \ref{t1.1}, Theorem \ref{t1.3} and Theorem \ref{t1.4} respectively. 
Note that Theorem \ref{t1.4} enables us to get lifts of 
multiplication matrices of Gaussian periods within the base field $\Q$ 
although Davenport and Hasse's lifting theorem for Jacobi sums 
needs to consider the extended field $\Q(\zeta_l)$. 
In particular, we can recover results \cite[Theorem 1, Corollary 9]{Hos06} 
for the reduced period polynomial $P_{5,5s}^*(X)$ and 
the exponential Gauss sums $g_{5s}(5)$ (see (2) below). 

Let $p$ be a prime with $p\equiv 1\ ({\rm mod}\ l)$. 
Katre and Rajwade \cite[Main theorem, page 186]{KR85a} gave 
some system of Diophantine equations whose unique solution gives 
the coefficient $a_1(n),\ldots,a_{l-1}(n)$ of the Jacobi sums 
$J_r(1,n)=\sum_{i=1}^{l-1}a_i(n)\zeta_l^i \in \Z[\zeta_l]$ 
$(1\leq n\leq l-2)$ 
and the cyclotomic numbers ${\rm Cyc}_{r}(i,j)$ of order $l$ 
are obtained in terms of $a_1(n),\ldots,a_{l-1}(n)$ $(1\leq n\leq l-2)$ as 
\begin{align*}
l^2{\rm Cyc}_{r}(0,0)&=p^r-3l+1- \sum_{n=1}^{l-2}\sum_{k=1}^{l-1}a_k(n),\\
l^2{\rm Cyc}_{r}(i,j)&=l\left(\delta_{i,0}+\delta_{0,j}+\delta_{i,j}+\sum_{n=1}^{l-2}a_{in+j}(n)\right)+l^2{\rm Cyc}_{r}(0,0) 
\end{align*}
where the subscripts in $a_{in+j} (n)$ are considered modulo $l$. 
See also van Wamelen \cite{Wam02} 
for general cases where $e\geq3$ and $p^r\equiv 1\ ({\rm mod}\ e)$. 
Recall that $C_r = [{\rm Cyc}_r(i,j) - D_i f]_{0\leq i,j\leq e-1}$ is 
the multiplication matrix of the Gaussian periods 
$\eta_r(0),\ldots,\eta_r(e-1)$ of degree $e$ for $\F_{p^{r}}$. 
Hence Theorem \ref{t1.4} gives explicit lifts of 
not only 
the multiplication matrix $C_r$ 
but also of 
cyclotomic numbers ${\rm Cyc}_r(i,j)$ and of Jacobi sums $J_r(i,j)$ 
from $\F_{p^r}$ to $\F_{p^{nr}}$. 
Recall also that the Jacobi sum $J_r(i,j)=J_r(\chi^i,\chi^j)
=\sum_{\alpha \in \F_{p^r}}\chi^i(\alpha)\chi^j(1-\alpha)$
where $\chi$ is a character of order $e$  on $\F_{p^r}$ 
with $\chi(\gamma)=\zeta_e$, 
$\F_{p^r}^\times=\langle\gamma\rangle$ and $\chi(0)=0$. 

Computing exponential Gauss sums $g_{er}(\lambda^i,e)=\eta_{er}^*(i)$ 
($\F_{p^{er}}^\times=\langle\lambda\rangle$) 
%(resp. $\eta_{er}(i)$, $G_{er}^*(\chi^i)$) 
is important because it is equivalent to that of 
the weight distribution of irreducible cyclic codes 
(see McEliece and Rumsey \cite{MR72}, 
Baumert and McEliece \cite[Theorem 6]{BM72}, 
McEliece \cite[Section 2]{McE74}, 
Ding and Yang \cite[Section 3]{DY13} and \cite[Section 11.7]{BEW98}). 

We give applications of Theorem \ref{t1.4} 
for prime degree $e=l$ with $3\leq l\leq 23$.\\

%%%%%%%%%%%%%%%%%%%%%%%%%%%%%%%%%%%%%%%%%%%%%%%%%%%%%%%%%%%%%%%%%%%
(1) $e=3$ (cf. Gauss \cite[Section 358]{Gau01}, Katre and Rajwade \cite[Proposition 1]{KR85a}, \cite[Section 3.1, Section 10.10]{BEW98}). 
Let $p$ be a prime with $p\equiv 1\ ({\rm mod}\ 3)$. 
The Jacobi sum $J_r(1,1)$ is given by 
\begin{align*}
J_r(1,1)=J(c,d)=\frac{c+3d}{2}+3d\zeta_3
\end{align*}
where $c,d\in\Z$ are given as the integer solutions of 
the Diophantine equation
\begin{align}
\begin{cases}
4p^r=c^2+27d^2,\\
c\equiv 1\ ({\rm mod}\ 3),\ p\nmid c.
\end{cases}\label{Eq3}
\end{align}
The equations have two solutions $(c,\pm d)$ 
and the sign of $d$ depends on the choice of $\gamma$. 
The unique solution with respect to $\gamma$ can be determined by 
\begin{align*}
\gamma^{(p^r-1)/3}\equiv\frac{c+9d}{c-9d}\ ({\rm mod}\ p)
\end{align*}
(see Katre and Rajwade \cite[Proposition 1]{KR85a}, \cite[Section 3.1]{BEW98}). 
The multiplication matrix $C_r$ of the Gaussian periods 
$\eta_r(0),\eta_r(1),\eta_r(2)$ of degree $3$ for 
$\F_{p^r}$ is given by
\begin{align*}
C_r=C_r(p,c,d)=
\left(
\begin{array}{ccccc}
 A-f & B-f & C-f\\
 B & C & D\\
 C & D & B\\
\end{array}
\right)
\end{align*}
where 
\begin{align*}
A&={\rm Cyc}_r(0,0)=\frac{1}{9}(p^r+c-8),& 
B&={\rm Cyc}_r(0,1)=\frac{1}{18}(2p^r-c+9d-4),\\ 
C&={\rm Cyc}_r(0,2)=\frac{1}{18}(2p^r-c-9d-4),& 
D&={\rm Cyc}_r(1,2)=\frac{1}{9}(p^r+c+1)
\end{align*}
(see \cite[Section 2.3]{BEW98}). 
Then we have 
\begin{align*}
P_{3,r}(X)&={\rm Char}_X(C_r(p,c,d)),\\
P_{3,r}^*(X)&=P_{3,r}^*(p,c,d;X)=3^3P_{3,r}((X-1)/3)=X^3-3p^rX-p^rc.
\end{align*}

By Theorem \ref{t1.4}, we get 
\begin{align*}
C_{nr}&=C_{nr}(p,c,d)=(-1)^{n-1}C_r(p,c,d)^{(n)}
=C_r(p^n,c^{(n)},d^{(n)})
\end{align*}
where $c^{(n)}$, $d^{(n)}$ can be obtained as a form of degree $n$ 
in $c$, $d$. 
Note that $c^{(n)}$, $d^{(n)}$ satisfy 
the equation (\ref{Eq3}) with respect to $p^{nr}$ 
instead of $p^r$. 
In particular, we get $J_{nr}(1,1)=J(c^{(n)},d^{(n)})$ and 
\begin{align*}
p^{nr}=\left(\frac{c^2+27d^2}{4}\right)^n
=\frac{(c^{(n)})^2+27(d^{(n)})^2}{4}.
\end{align*}
For $n=2,3$, we can obtain that 
\begin{align*}
c^{(2)}&=\frac{1}{2}(-c^2+27d^2),\ d^{(2)}=-cd,\\
c^{(3)}&=\frac{1}{4}c(c+9d)(c-9d),\ d^{(3)}=\frac{3}{4}d(c+3d)(c-3d).
\end{align*}
and 
\begin{align*}
P_{3,2r}^*(X)=P_{3,r}^*(p^2,c^{(2)},d^{(2)};X)
&=X^3-3p^{2r}X-p^{2r}c^{(2)}\\
&=X^3-3p^{2r}X-\frac{1}{2}p^{2r}(-c^2+27d^2),\\
P_{3,3r}^*(X)=P_{3,r}^*(p^3,c^{(3)},d^{(3)};X)
&=X^3-3p^{3r}X-p^{3r}c^{(3)}\\
&=X^3-\frac{3}{4}p^{2r}(c^2+27d^2)X-\frac{1}{4}\left(p^{3r}c(c+9d)(c-9d)\right)\\
&=\left(X-p^rc\right)\left(X+p^r\,\frac{c+9d}{2}\right)\left(X+p^r\,\frac{c-9d}{2}\right)
\end{align*}
(see \cite[Section 12.10, page 427]{BEW98}). 
Because the exponential cubic Gauss sum 
$g_{3r}(3)=\eta_{3r}^*(0)$ is one of the roots 
of $P_{3,3r}^*(X)$ which 
does not depend on the choice of $\gamma$, 
we get 
\begin{align*}
g_{3r}(3)=p^rc. 
\end{align*}
Similarly, for $n=3m$, we have 
\begin{align*}
g_{3mr}(3)&=p^{mr}c^{(m)}. 
\end{align*}

%%%%%%%%%%%%%%%%%%%%%%%%%%%%%%%%%%%%%%%%%%%%%%%%%%%%%%%%%%%%%%%%%%%
\medskip
(2) $e=5$ (cf. Lehmer \cite[Equation (10)]{Leh51}, Berndt and Evans \cite[Section 5]{BE81},  Katre and Rajwade \cite{KR85b}, \cite[Section 3.7]{BEW98}, Hoshi \cite[Section 5]{Hos03}, \cite[Section 3]{Hos06}).
Let $p$ be a prime with $p\equiv 1\ ({\rm mod}\ 5)$. 
The Jacobi sum $J_r(1,1)$ is given by 
\begin{align*}
J_r(1,1)=J(x,w,v,u)=
\frac{1}{4}\left(Z\zeta_5+\sigma^3(Z)\zeta_5^2+\sigma(Z)\zeta_5^3+\sigma^2(Z)\zeta_5^4\right)
\end{align*}
where $Z=-x+5w+4v+2u$, $\sigma(x,w,v,u)=(x,-w,-u,v)$ and
$x,w,v,u\in\Z$ are obtained as the integer solutions of 
the system of Diophantine equations 
\begin{align}
\begin{cases}
16p^r=x^2+125w^2+50v^2+50u^2,\\
xw=v^2-4vu-u^2, \\
x \equiv 1\ ({\rm mod}\ 5),\ p\nmid x^2-125w^2.
\end{cases}\label{Eq5}
\end{align}
The equations have four solutions $\sigma^i(x,w,v,u)$ $(i=0,1,2,3)$ 
which depend on the choice of $\gamma$. 
The unique solution with respect to $\gamma$ can be determined by 
\begin{align*}
\gamma^{(p^r-1)/5}\equiv\frac{x^2-125w^2-10(2xu-xv-25wv)}{x^2-125w^2+10(2xu-xv-25wv)}\ ({\rm mod}\ p)
\end{align*}
(see Katre and Rajwade \cite[Theorem 1]{KR85b}). 
The multiplication matrix of $C_r$ of the Gaussian periods 
$\eta_{r}(0),\ldots,\eta_{r}(4)$ of order $5$ for $\F_{p^{r}}$ is given by
\begin{align*}
C_r=C_r(p,x,u,v,w)=
\left(
\begin{array}{ccccc}
 A-f & B-f & C-f & D-f & E-f \\
 B & E & F & G & F \\
 C & F & D & G & G \\
 D & G & G & C & F \\
 E & F & G & F & B \\
\end{array}
\right)
\end{align*}
where 
\begin{align*}
A&=\frac{1}{25}(p^r+3x-14),\\
B&=\frac{1}{100}(4p^r-3x+25w+50v-16),& 
C&=\frac{1}{100}(4p^r-3x-25w+50u-16),\\
D&=\frac{1}{100}(4p^r-3x-25w-50u-16),& 
E&=\frac{1}{100}(4p^r-3x+25w-50v-16),\\
F&=\frac{1}{50}(2p^r+x-25w+2),& 
G&=\frac{1}{50}(2p^r+x+25w+2).
\end{align*}
Then we have 
\begin{align*}
P_{5,r}(X)&={\rm Char}_X(C_r(p,x,u,v,w)),\\
P_{5,r}^*(X)&=P_{5,r}^*(p,x,w,v,u;X)=5^5P_{5,r}((X-1)/5)\\
&=X^5-10p^rX^3-5p^rxX^2
+\frac{5}{4}p^r(4p^r-x^2+125w^2)X+\frac{1}{8}p^r(-x^3+8p^rx+625w(v^2-u^2)) 
\end{align*}
(see \cite[Equation (10)]{Leh51}, \cite[Section 5]{BE81}, \cite[Section 3]{Hos06}). 

By Theorem \ref{t1.4}, we get 
\begin{align*}
C_{nr}=C_{nr}(p,x,w,v,u)=(-1)^{n-1}C_r(p,x,w,v,u)^{(n)}
=C_r(p^n,x^{(n)},w^{(n)},v^{(n)},u^{(n)})
\end{align*}
where  
$x^{(n)}$, $w^{(n)}$, $v^{(n)}$, $u^{(n)}$ 
can be obtained as a form of degree $n$ 
in $x$, $w$, $v$, $u$. 
Note that $x^{(n)}$, $w^{(n)}$, $v^{(n)}$, $u^{(n)}$ 
satisfy the equation (\ref{Eq5}) with respect to 
$p^{nr}$ instead of $p^r$. 
In particular, we have 
$J_{nr}(1,1)=J(x^{(n)},w^{(n)},v^{(n)},u^{(n)})$,  
\begin{align*}
p^{nr}&=\left(\frac{x^2+125w^2+50v^2+50u^2}{16}\right)^n
=\frac{(x^{(n)})^2+125(w^{(n)})^2+50(v^{(n)})^2+50(u^{(n)})^2}{16},\\
x^{(n)}w^{(n)}&=(v^{(n)})^2-4v^{(n)}u^{(n)}-(u^{(n)})^2.
\end{align*}
For $n=2$, we can obtain that 
\begin{align*}
x^{(2)}&=\frac{1}{4}(-x^2-125w^2+50v^2+50u^2),\\
w^{(2)}&=\frac{1}{2}(-xw-v^2+4vu+u^2),\\
v^{(2)}&=\frac{1}{2}(-xv-10wu+5vw),\\
u^{(2)}&=\frac{1}{2}(-xu-10wv-5uw).
\end{align*} %
Continuing the argument, we get 
$x^{(5)}$, $w^{(5)}$, $v^{(5)}$, $u^{(5)}$ and hence 
$P_{5,5r}(X)={\rm Char}_X(C_r^{(5)})$ and 
\begin{align*}
P_{5,5r}^*(X)
=\left(X-\frac{p^r}{16}L\right)
\prod_{l=0}^3\left(X-\frac{p^r}{64}\sigma^l(M)\right)
\end{align*}
where 
\begin{align*}
L=L(x,w,v,u)=&\ x^3-50(v^2+u^2)w-125(11v^2-4vu-11u^2),\\
M=M(x,w,v,u)=&-x^3+25x\left(2ux+(7v-u)(v+3u)\right)\\
&+125w\left(25w^2+10(4v-3u)w+(7v-u)(v+3u)\right)\\
&+500(-2v^3-3v^2u+6vu^2+u^3)
\end{align*}
(see \cite[Theorem 1]{Hos06}). 
Because the exponential quintic Gauss sum 
$g_{5r}(5)=\eta_{5r}^*(0)$ is one of the roots 
of $P_{5,5r}^*(X)$ which 
does not depend on the choice of $\gamma$, 
we have 
\begin{align*}
g_{5r}(5)=\frac{p^r}{16}L(x,w,v,u) 
\end{align*}
(see \cite[Corollary 9]{Hos06}).
Similarly, for $n=5m$, we have 
\begin{align*}
g_{5mr}(5)&=\frac{p^{mr}}{16}L(x^{(m)},w^{(m)},v^{(m)},u^{(m)}). 
\end{align*}

For example, we take $p=11$ and $r=1$. 
Then we have $g_5(5)=\eta^*_{5}(0)=\frac{11}{16}L(1,1,1,0)=-979=-11\cdot 89$. 
Indeed, we may check that 
$g_5(5)=13751\cdot \zeta_{11}^0
+14730\cdot \sum_{i=1}^{10}\zeta_{11}^i=13751-14730=-979$ 
by the definition using a computer (cf. Section \ref{S5} (2) $e=5$).\\

%%%%%%%%%%%%%%%%%%%%%%%%%%%%%%%%%%%%%%%%%%%%%%%%%%%%%%%%%%%%%%%%%%%
(3) $e=7$ (cf. Leonard and Williams \cite{LW75}, \cite[Section 3.9]{BEW98}). 
Let $p$ be a prime with $p\equiv 1\ ({\rm mod}\ 7)$. 
The Jacobi sums $J_r(1,1)$ and $J_r(1,2)$ are given by 
\begin{align*}
J_r(1,1)&=J(x_1,x_2,x_3,x_4,x_5,x_6)\\
&=\frac{1}{12}\left(Z\zeta_7+\sigma^4(Z)\zeta_7^2+\sigma^5(Z)\zeta_7^3
+\sigma^2(Z)\zeta_7^4+\sigma(Z)\zeta_7^5+\sigma^3(Z)\zeta_7^6\right),\\
J_r(1,2)&=J^\prime(t,u)=-t+u\sqrt{-7}
\end{align*}
where $Z=-2x_1+6x_2+7x_5+21x_6$, 
$\sigma(x_1,x_2,x_3,x_4,x_5,x_6)=(x_1,-x_3,x_4,x_2,(-x_5-3x_6)/2,(x_5-x_6)/2)$ 
and 
$x_1,x_2,x_3,x_4,x_5,x_6,t,u\in\Z$ are obtained as 
the integer solutions of the system of Diophantine equations 
\begin{align}
\begin{cases}
72p^r=2x_1^2+42x_2^2+42x_3^2+42x_4^2+343x_5^2+1029x_6^2, \\
12x_2^2-12x_4^2+147x_5^2-441x_6^2+56x_1x_6+24x_2x_3
-24x_2x_4+48x_3x_4+98x_5x_6=0, \\
12x_3^2-12x_4^2+49x_5^2-147x_6^2+28x_1x_5+28x_1x_6
+48x_2x_3+24x_2x_4+24x_3x_4\\
+490x_5x_6=0,\\
x_1\equiv 1\ ({\rm mod}\ 7),\ (x_5,x_6)\neq (0,0),\\
p^r=t^2+7u^2,\ t\equiv 1\ ({\rm mod}\ 7),\ u\equiv 3x_2+2x_3\ ({\rm mod}\ 7).
\end{cases}\label{Eq7}
\end{align}
The equations has six solutions 
$\sigma^i(x_1,x_2,x_3,x_4,x_5,x_6)$ $(i=0,1,2,3,4,5)$ 
which depend on the choice of $\gamma$. 
The multiplication matrix $C_r$ of the Gaussian periods 
$\eta_{r}(0),\ldots,\eta_{r}(6)$ of order $7$ for $\F_{p^{r}}$ is given by
\begin{align*}
C_r
&=\,C_r(p,x_1,x_2,x_3,x_4,x_5,x_6,t,u)\\
&=
\left(
\begin{array}{ccccccc}
 A-f & B-f & C-f & D-f & E-f & F-f & G-f \\
 B & G & H & I & J & K & H \\
 C & H & F & K & L & L & I \\
 D & I & K & E & J & L & J \\
 E & J & L & J & D & I & K \\
 F & K & L & L & I & C & H \\
 G & H & I & J & K & H & B \\
\end{array}
\right)
\end{align*}
where 
\begin{align*}
A&=\frac{1}{49}(p^r+3x_1-12t-20),\\
B&=\frac{1}{196}(4p^r-2x_1+28x_2-14x_3+49x_5+49x_6+8t+56u-24),\\
C&=\frac{1}{98}(2p^r-x_1+14x_3+7x_4-49x_6+4t+28u-12),\\
D&=\frac{1}{196}(4p^r-2x_1+14x_2+28x_4-49x_5+49x_6+8t-56u-24),\\
E&=\frac{1}{196}(4p^r-2x_1-14x_2-28x_4-49x_5+49x_6+8t+56u-24),\\
F&=\frac{1}{98}(2p^r-x_1-14x_3-7x_4-49x_6+4t-28u-12),\\
G&=\frac{1}{196}(4p^r-2x_1-28x_2+14x_3+49x_5+49x_6+8t-56u-24),\\
H&=\frac{1}{147}(3p^r+2x_1-49x_5+6t+3),\\
I&=\frac{1}{98}(2p^r-x_1+7x_2+7x_3-7x_4-10t-14u+2),\\
J&=\frac{1}{294}(6p^r+4x_1+49x_5-147x_6+12t+6),\\
K&=\frac{1}{98}(2p^r-x_1-7x_2-7x_3+7x_4-10t+14u+2),\\
L&=\frac{1}{294}(6p^r+4x_1+49x_5+147x_6+12t+6)
\end{align*}
(see \cite[Theorem]{LW75} with a typo for $B$ 
($147 x_4$ should be $147 x_5$)).
Then we have 
\begin{align*}
P_{7,r}(X)&={\rm Char}_X(C_r(p,x_1,x_2,x_3,x_4,x_5,x_6,t,u)),\\
P_{7,r}^*(X)&=P^*_{7,r}(p,x_1,x_2,x_3,x_4,x_5,x_6,t,u;X)=7^7P_{7,r}((X-1)/7)\\
&=X^7-21p^rX^5+a_4X^4+a_3X^3+a_2X^2+a_1X+a_0
\end{align*}
where $a_i=a_i(p,x_1,x_2,x_3,x_4,x_5,x_6,t,u)$ $(0\leq i\leq 4)$ 
can be obtained explicitly (we omit the display here). 

By Theorem \ref{t1.4}, we get 
\begin{align*}
C_{nr}=(-1)^{n-1}C_r(p,x_1,x_2,x_3,x_4,x_5,x_6,t,u)^{(n)}
=C_r(p^n,x_1^{(n)},x_2^{(n)},x_3^{(n)},x_4^{(n)},x_5^{(n)},x_6^{(n)},t^{(n)},u^{(n)})
\end{align*}
where $x_1^{(n)}$, $x_2^{(n)}$, $x_3^{(n)}$, 
$x_4^{(n)}$, $x_5^{(n)}$, $x_6^{(n)}$, $t^{(n)}$, $u^{(n)}$ 
satisfy the equation (\ref{Eq7}) with respect to $p^{nr}$ instead of $p^r$. 
In particular, we have 
$J_{nr}(1,1)=J(x_1^{(n)},x_2^{(n)},x_3^{(n)},x_4^{(n)},x_5^{(n)},x_6^{(n)})$, 
$J_{nr}(1,2)=J^\prime(t^{(n)},u^{(n)})$ and 
\begin{align*}
p^{nr}&=\left(\frac{2x_1^2+42x_2^2+42x_3^2+42x_4^2+343x_5^2+1029x_6^2}{72}\right)^n\\
&=\frac{2(x_1^{(n)})^2+42(x_2^{(n)})^2+42(x_3^{(n)})^2+42(x_4^{(n)})^2+343(x_5^{(n)})^2+1029(x_6^{(n)})^2}{72},\\
p^{nr}&=(t^2+7u^2)^n=(t^{(n)})^2+7(u^{(n)})^2.
\end{align*}
For $n=2$, we get 
\begin{align*}
x_1^{(2)} &= \frac{1}{12}(-2x_1^2 + 42x_2^2 + 42x_3^2 + 42x_4^2 - 343x_5^2 -1029x_6^2),\\
x_2^{(2)} &= \frac{1}{12}(-4x_1x_2 + 14x_2x_5 - 42x_3x_5 - 42x_4x_5 - 42x_2x_6 - 42x_3x_6 + 42x_4x_6),\\
x_3^{(2)} &= \frac{1}{12}(-4x_1x_3 - 42x_2x_5 - 28x_3x_5 - 42x_2x_6 - 84x_4x_6),\\
x_4^{(2)} &= \frac{1}{12}(-4x_1x_4 - 42x_2x_5 + 14x_4x_5 + 42x_2x_6 - 84x_3x_6 + 42x_4x_6),\\
x_5^{(2)} &= \frac{1}{168}(-12x_2^2 + 72x_2x_3 + 24x_3^2 + 72x_2x_4 - 12x_4^2 - 56x_1x_5 + 49x_5^2 - 882x_5x_6 - 147x_6^2),\\
x_6^{(2)} &= \frac{1}{168}(12x_2^2 + 24x_2x_3 - 24x_2x_4 + 48x_3x_4 - 12x_4^2 - 147x_5^2 - 56x_1x_6 - 98x_5x_6 + 441x_6^2),\\
t^{(2)}&=t^2-7u^2,\\
u^{(2)}&=2tu. 
\end{align*}

For example, we take $p=29=ef+1$ with $f=4$. 
Then we get 
\begin{align*}
P_{7,7}^*(X)
=&\ (X-442569)(X-408465)(X-233682)(X+182671)\\
&\cdot(X+259405)(X+317869)(X+324771).
\end{align*}
Indeed, we may check that 
$g_7(7)=594516413\cdot \zeta_{29}^0
+594834282\cdot \sum_{i=1}^{28}\zeta_{29}^i=594516413-594834282=-317869=-29\cdot 97\cdot 113$ 
by the definition using a computer 
(cf. Section \ref{S5} (3) $e=7$).\\

(4) $e=11$, $e=13$ and $e=17$. 
By using Thaine's formula \cite[page 259]{Tha04}, 
we can obtain the multiplication matrix $C_1$ of 
the Gaussian periods
$\eta_{1}(0),\ldots,\eta_{1}(e-1)$ of degree $e$ for $\F_{p^1}$.
By using Theorem \ref{t1.4} as in the case of $e=7$, 
we get $C_{e}=C_{1}^{(e)}$, $P_{e,e}(X)={\rm Char}_X(C_1^{(e)})$
and the explicit factorization of $P_{e,e}^*(X)$ into $e$ linear factors. 
For example, we take $p=23=ef+1$ with $e=11$ and $f=2$. 
Then, we get $C_{11}=C_{1}^{(11)}$, 
$P_{11,11}(X)={\rm Char}_X(C_1^{(11)})$ and 
\begin{align*}
P_{11,11}^*(X)=
&\ (X-166665038)(X-142959444)(X-52918009)(X-47121273)\\
&\cdot (X-3199967)(X+19592803)(X+23093817)(X+44439427)\\
&\cdot (X+58652208)(X+64390754)(X+202694722)
\end{align*}
with a root $g_{11}(11)=\eta^*_{11}(0)=52918009=23\cdot 53\cdot 43441$ 
(cf. Section \ref{S5} (4) $e=11$). 
%[ 23, 53, 43411 ]
Similarly, by using Theorem \ref{t1.4}, 
we can obtain for $p=53=ef+1$ with $e=13$ and $f=4$,
$P_{13,13}^*(X)$ with a root $g_{13}(13)=\eta^*_{13}(0)=782475795674
=2\cdot 53\cdot 7381847129$ and 
%[ 2, 53, 7381847129 ]
for $p=103=ef+1$ with $e=17$ and $f=6$, 
$P_{17,17}^*(X)$ with a root $g_{17}(17)=\eta^*_{17}(0)=-651513206543247755
=-5\cdot 7\cdot 103\cdot 172709\cdot 1046412659$ 
%[ 5, 7, 103, 172709, 1046412659 ]
(cf. Section \ref{S5} (5) $e=13$, (6) $e=17$).\\

(5) $e=19$ and $e=23$. We take $p=191=ef+1$ with $e=19$ and $f=10$. 
As in the case (4), we can get 
$C_{19}=C_{1}^{(19)}$, 
$P_{19,19}(X)={\rm Char}_X(C_1^{(19)})$ 
and  the explicit factorization of 
$P_{19,19}^*(X)$ into $19$ linear factors. 
We see that $\eta^*_{19}(i)=p\,\xi_i$ 
and the $\xi_i$'s $(i\neq 0)$ are permuted under the action 
$\zeta_p\mapsto \zeta_p^\gamma$ with $\F_p^\times=\langle\gamma\rangle$
(which depends on the choice of $\gamma$). 
Because we see that 
$\xi_0\in \F_{191}^\times$ is of order $10$  
and $\xi_i \in \F_{191}^\times$ $(1\leq i\leq 18)$ is of order $190$, 
we can find 
$p\,\xi_0=g_{19}(19)=\eta^*_{19}(0)=2801935824159299141695
=5\cdot 191\cdot 509\cdot 26374987\cdot 218546963$. 
%[ 5, 191, 509, 26374987, 218546963 ]
Similarly, for $p=47=ef+1$ with $e=23$ and $f=2$, 
we get $P^*_{23,23}(X)$ with a root 
$g_{23}(23)=\eta^*_{23}(0)=-492643134044787602=-2\cdot 17\cdot 
43\cdot 47\cdot 7169472509893$. 
%[ -2, 17, 43, 47, 7169472509893 ]

%~\\
%%%%%%%%%%%%%%%%%%%%%%%%%%%%%%%%%%%%%%%%%%%%%%%%%%%%%%%%%%%%%%%%%%%%
%We give GAP (\cite{GAP}) computations for examples above 
%in the arXiv version of this paper \cite[Section \ref{S7}]{HK}.
%%%%%%%%%%%%%%%%%%%%%%%%%%%%%%%%%%%%%%%%%%%%%%%%%%%%%%%%%%%%%%%%%%%%
~\\
We give GAP (\cite{GAP}) computations for examples above. 
The function ${\tt MultMat}(e,p,\gamma)$ returns 
the multiplication matrix $C_1$ of the Gaussian periods 
$\eta_{1}(0),\ldots,\eta_{1}(e-1)$ of degree $e$ for $\F_{p^1}$ 
with respect to the generator $\gamma$ of $\F_p^\times$
using Thaine's formula \cite[page 259]{Tha04}. 
The function ${\tt dComp}(A,B,d)$ 
returns the $d$-composition $A\overset{d}{\ast}B$ 
for two matrices $A$ and $B$.\\ 
%%%%%%%%%%%%%%%%%%%%%%%%%%%%%%%%%%%%%%%%%%%%%%%%%%%%%%%%%%%%%%%%%%%

\begin{verbatim}
MultMat:=function(e,p,g)
  local f,mat,j; 
  f:=(p-1)/e;
  mat:=List([0..e-1],i->List([0..e-1],j->(-1/e^2)*Sum([0..e-1],
    l->Sum([0..e],k->Binomial(f*k,f*l)*g^(f*(l*i-k*j)))) mod p));
  if IsEvenInt(f) then for j in [1..e] do mat[1,j]:=mat[1,j]-f;od; 
    else for j in [1..e] do mat[e/2+1,j]:=mat[e/2+1,j]-f;od;
  fi;
  return mat;
end;

Mode:=function(a,e)
  if a mod e = 0 then return e; else return a mod e;fi;
end;

dComp:=function(A,B,d)
  local s,t,i,j,e,mat;
  if Size(A)=Size(B) then e:=Size(A); else return "Input error";
  fi;
  mat:=List([0..e-1],i->List([0..e-1],j->Sum([0..e-1],
    s->Sum([0..e-1],t->A[s+1,t+1]*B[Mode(d*s+i+1,e),Mode(d*t+j+1,e)]))));
  return mat;
end;

gap> PrimitiveRootMod(7); # g=3
3
gap> C:=MultMat(3,7,3); # Multiplication matrix C1 for e=3, p=7, g=3
[ [ -2, -2, -1 ], 
  [ 0, 1, 1 ], 
  [ 1, 1, 0 ] ]
gap> C2:=dComp(C,C,-1); # C2=C^(2)
[ [ 10, 11, 12 ], 
  [ -5, -4, -7 ], 
  [ -4, -7, -5 ] ]
gap> C3:=dComp(C2,C,-1); # C3=C^(3)
[ [ -79, -72, -78 ], 
  [ 42, 36, 36 ], 
  [ 36, 36, 42 ] ]
gap> P3:=CharacteristicPolynomial(C3); # P3 is the priod polynomial for r=3
x_1^3+x_1^2-114*x_1+216
gap> R3:=RootsOfPolynomial(P3); # roots of P3
[ 9, 2, -12 ]
gap> L3:=List(R3,x->3*x+1); # roots of the reduced period polynomial P3^*
[ 28, 7, -35 ]
gap> X3:=L3/7;
[ 4, 1, -5 ]
gap> List(X3,x->x mod 7); # X3[2]=1 mod 7
[ 4, 1, 2 ]
gap> List(X3,x->x^2 mod 7);
[ 2, 1, 4 ]
gap> List(X3,x->x^3 mod 7); # X3[i] (i<>2) is of order 3 in F7^x
[ 1, 1, 1 ]
gap> L3[2]; # L3[2] is the exponential Gauss sum g_3(3)
7

gap> PrimitiveRootMod(11); # g=2
2
gap> C:=MultMat(5,11,2); # Multiplication matrix C1 for e=5, p=11, g=2
[ [ -2, -1, -2, -2, -2 ], 
  [ 1, 0, 0, 1, 0 ], 
  [ 0, 0, 0, 1, 1 ], 
  [ 0, 1, 1, 0, 0 ], 
  [ 0, 0, 1, 0, 1 ] ]
gap> C2:=dComp(C,C,-1);; # C2=C^(2)
gap> C4:=dComp(C2,C2,-1);; # C4=C^(4)
gap> C5:=dComp(C4,C,-1); # C5=C^(5)
[ [ -25721, -25790, -25680, -25830, -25820 ], 
  [ 6420, 6390, 6500, 6400, 6500 ], 
  [ 6530, 6500, 6380, 6400, 6400 ], 
  [ 6380, 6400, 6400, 6530, 6500 ], 
  [ 6390, 6500, 6400, 6500, 6420 ] ]
gap> P5:=CharacteristicPolynomial(C5); # P5 is the priod polynomial for r=5
x_1^5+x_1^4-64420*x_1^3-2589700*x_1^2+558588000*x_1+11695320000
gap> R5:=RootsOfPolynomial(P5); # roots of P5
[ 255, 90, -20, -130, -196 ]
gap> L5:=List(R5,x->5*x+1); # roots of the reduced period polynomial P5^*
[ 1276, 451, -99, -649, -979 ]
gap> X5:=L5/11;
[ 116, 41, -9, -59, -89 ]
gap> List(X5,x->x mod 11); # X5[5]=-1 mod 11
[ 6, 8, 2, 7, 10 ]
gap> List(X5,x->x^2 mod 11);
[ 3, 9, 4, 5, 1 ]
gap> List(X5,x->x^5 mod 11); # X5[i] (i<>5) is of order 10 in F11^x
[ 10, 10, 10, 10, 10 ]
gap> L5[5]; # L5[5] is the exponential Gauss sum g_5(5)
-979
gap> Factors(L5[5]);
[ -11, 89 ]

gap> PrimitiveRootMod(29); # g=2
2
gap> C:=MultMat(7,29,2); # Multiplication matrix C1 for e=7, p=29, g=2
[ [ -4, -3, -4, -4, -2, -4, -4 ], 
  [ 1, 0, 1, 0, 0, 1, 1 ], 
  [ 0, 1, 0, 1, 1, 1, 0 ], 
  [ 0, 0, 1, 2, 0, 1, 0 ], 
  [ 2, 0, 1, 0, 0, 0, 1 ], 
  [ 0, 1, 1, 1, 0, 0, 1 ], 
  [ 0, 1, 0, 0, 1, 1, 1 ] ]
gap> C2:=dComp(C,C,-1);; # C2=C^(2)
gap> C4:=dComp(C2,C2,-1);; # C4=C^(4)
gap> C3:=dComp(C2,C,-1);; # C3=C^(3)
gap> C7:=dComp(C4,C3,-1);; # C7=C^(7)
gap> P7:=CharacteristicPolynomial(C7);; # P7 is the priod polynomial for r=7
gap> R7:=RootsOfPolynomial(P7);; # roots of P7
gap> L7:=List(R7,x->7*x+1); # roots of the reduced period polynomial P7^*
[ 442569, 408465, 233682, -182671, -259405, -317869, -324771 ]
gap> X7:=L7/29;
[ 15261, 14085, 8058, -6299, -8945, -10961, -11199 ]
gap> List(X7,x->x mod 29); # X7[6]=1 mod 29
[ 7, 20, 25, 23, 16, 1, 24 ]
gap> List(X7,x->x^2 mod 29);
[ 20, 23, 16, 7, 24, 1, 25 ]
gap> List(X7,x->x^4 mod 29);
[ 23, 7, 24, 20, 25, 1, 16 ]
gap> List(X7,x->x^7 mod 29); # X7[i] (i<>6) is of order 7 in F29^x
[ 1, 1, 1, 1, 1, 1, 1 ]
gap> L7[6]; # L7[6] is the exponential Gauss sum g_7(7)
-317869
gap> Factors(L7[6]);
[ -29, 97, 113 ]

gap> PrimitiveRootMod(23); # g=5
5
gap> C:=MultMat(11,23,5); # Multiplication matrix C1 for e=11, p=23, g=5
[ [ -2, -2, -1, -2, -2, -2, -2, -2, -2, -2, -2 ], 
  [ 0, 0, 0, 0, 1, 0, 0, 1, 0, 0, 0 ], 
  [ 1, 0, 0, 0, 0, 1, 0, 0, 0, 0, 0 ], 
  [ 0, 0, 0, 0, 0, 0, 0, 0, 0, 1, 1 ], 
  [ 0, 1, 0, 0, 0, 1, 0, 0, 0, 0, 0 ], 
  [ 0, 0, 1, 0, 1, 0, 0, 0, 0, 0, 0 ], 
  [ 0, 0, 0, 0, 0, 0, 0, 0, 1, 0, 1 ], 
  [ 0, 1, 0, 0, 0, 0, 0, 0, 1, 0, 0 ], 
  [ 0, 0, 0, 0, 0, 0, 1, 1, 0, 0, 0 ], 
  [ 0, 0, 0, 1, 0, 0, 0, 0, 0, 1, 0 ], 
  [ 0, 0, 0, 1, 0, 0, 1, 0, 0, 0, 0 ] ]
gap> C2:=dComp(C,C,-1);; # C2=C^(2)
gap> C4:=dComp(C2,C2,-1);; # C4=C^(4)
gap> C8:=dComp(C4,C4,-1);; # C8=C^(8)
gap> C3:=dComp(C2,C,-1);; # C3=C^(3)
gap> C11:=dComp(C8,C3,-1);; # C11=C^(11)
gap> P11:=CharacteristicPolynomial(C11);; # P11 is the priod polynomial for r=11
gap> R11:=RootsOfPolynomial(P11);; # roots of P11
gap> L11:=List(R11,x->11*x+1); # roots of the reduced period polynomial P11^*
[ 166665038, 142959444, 52918009, 47121273, 3199967, -19592803, 
  -23093817, -44439427, -58652208, -64390754, -202694722 ]
gap> X11:=L11/23;;
gap> List(X11,x->x mod 23); # X11[3]=1 mod 23
[ 18, 16, 1, 3, 2, 13, 9, 12, 6, 8, 4 ]
gap> List(X11,x->x^2 mod 23);
[ 2, 3, 1, 9, 4, 8, 12, 6, 13, 18, 16 ]
gap> List(X11,x->x^11 mod 23); # X13[i] (i<>3) is of order 11 in F23^x
[ 1, 1, 1, 1, 1, 1, 1, 1, 1, 1, 1 ]
gap> L11[3]; # L11[3] is the exponential Gauss sum g_{11}(11)
52918009
gap> Factors(L11[3]);
[ 23, 53, 43411 ]

gap> PrimitiveRootMod(53); # g=2
2
gap> C:=MultMat(13,53,2); # Multiplication matrix C1 for e=13, p=53, g=2
[ [ -4, -3, -4, -4, -4, -4, -4, -2, -4, -4, -4, -4, -4 ], 
  [ 1, 0, 0, 1, 1, 1, 0, 0, 0, 0, 0, 0, 0 ], 
  [ 0, 0, 0, 0, 1, 0, 0, 0, 1, 0, 0, 1, 1 ],
  [ 0, 1, 0, 0, 0, 0, 0, 1, 1, 0, 0, 0, 1 ], 
  [ 0, 1, 1, 0, 0, 0, 0, 0, 0, 0, 1, 0, 1 ], 
  [ 0, 1, 0, 0, 0, 0, 0, 1, 1, 0, 1, 0, 0 ], 
  [ 0, 0, 0, 0, 0, 0, 2, 0, 0, 1, 0, 1, 0 ], 
  [ 2, 0, 0, 1, 0, 1, 0, 0, 0, 0, 0, 0, 0 ], 
  [ 0, 0, 1, 1, 0, 1, 0, 0, 0, 1, 0, 0, 0 ], 
  [ 0, 0, 0, 0, 0, 0, 1, 0, 1, 0, 1, 1, 0 ], 
  [ 0, 0, 0, 0, 1, 1, 0, 0, 0, 1, 0, 1, 0 ], 
  [ 0, 0, 1, 0, 0, 0, 1, 0, 0, 1, 1, 0, 0 ], 
  [ 0, 0, 1, 1, 1, 0, 0, 0, 0, 0, 0, 0, 1 ] ]
gap> C2:=dComp(C,C,-1);; # C2=C^(2)
gap> C4:=dComp(C2,C2,-1);; # C4=C^(4)
gap> C8:=dComp(C4,C4,-1);; # C8=C^(8)
gap> C5:=dComp(C4,C,-1);; # C5=C^(5)
gap> C13:=dComp(C8,C5,-1);; # C13=C^(13)
gap> P13:=CharacteristicPolynomial(C13);; # P13 is the priod polynomial for r=13
gap> R13:=RootsOfPolynomial(P13) # roots of P13
gap> L13:=List(R13,x->13*x+1); # roots of the reduced period polynomial P13^*
[ 1040615291340, 782475795674, 664438112586, 338244988654, 117899008800, 
  83828569254, -186980700750, -238169301889, -245670171356, -277653262665, 
  -427932303889, -740552966334, -910543059425 ]
gap> X13:=L13/53;;
gap> List(X13,x->x mod 53); # X13[2]=1 mod 53
[ 47, 1, 49, 16, 13, 15, 42, 36, 44, 46, 10, 28, 24 ]
gap> List(X13,x->x^2 mod 53);
[ 36, 1, 16, 44, 10, 13, 15, 24, 28, 49, 47, 42, 46 ]
gap> List(X13,x->x^4 mod 53);
[ 24, 1, 44, 28, 47, 10, 13, 46, 42, 16, 36, 15, 49 ]
gap> List(X13,x->x^13 mod 53); # X13[i] (i<>2) is of order 13 in F53^x
[ 1, 1, 1, 1, 1, 1, 1, 1, 1, 1, 1, 1, 1 ]
gap> L13[2]; # L13[2] is the exponential Gauss sum g_{13}(13)
782475795674
gap> Factors(L13[2]);
[ 2, 53, 7381847129 ]

gap> PrimitiveRootMod(103); # g=5
5
gap> C:=MultMat(17,103,5); # Multiplication matrix C1 for e=17, p=103, g=5
[ [ -4, -6, -6, -6, -6, -6, -6, -6, -6, -6, -5, -4, -6, -6, -6, -6, -6 ], 
  [ 0, 0, 1, 1, 0, 0, 0, 2, 0, 0, 0, 0, 0, 0, 0, 1, 1 ], 
  [ 0, 1, 0, 1, 0, 0, 1, 0, 0, 1, 0, 0, 0, 0, 1, 0, 1 ], 
  [ 0, 1, 1, 0, 0, 1, 0, 1, 0, 1, 0, 0, 1, 0, 0, 0, 0 ], 
  [ 0, 0, 0, 0, 0, 0, 0, 0, 1, 1, 0, 0, 1, 1, 1, 1, 0 ], 
  [ 0, 0, 0, 1, 0, 0, 0, 0, 1, 1, 1, 0, 1, 1, 0, 0, 0 ], 
  [ 0, 0, 1, 0, 0, 0, 2, 0, 0, 0, 0, 0, 0, 0, 1, 0, 2 ], 
  [ 0, 2, 0, 1, 0, 0, 0, 1, 0, 0, 0, 0, 1, 0, 0, 1, 0 ], 
  [ 0, 0, 0, 0, 1, 1, 0, 0, 0, 0, 1, 1, 1, 1, 0, 0, 0 ], 
  [ 0, 0, 1, 1, 1, 1, 0, 0, 0, 0, 0, 0, 0, 1, 1, 0, 0 ], 
  [ 1, 0, 0, 0, 0, 1, 0, 0, 1, 0, 0, 2, 0, 1, 0, 0, 0 ], 
  [ 2, 0, 0, 0, 0, 0, 0, 0, 1, 0, 2, 0, 0, 1, 0, 0, 0 ], 
  [ 0, 0, 0, 1, 1, 1, 0, 1, 1, 0, 0, 0, 0, 0, 0, 1, 0 ], 
  [ 0, 0, 0, 0, 1, 1, 0, 0, 1, 1, 1, 1, 0, 0, 0, 0, 0 ], 
  [ 0, 0, 1, 0, 1, 0, 1, 0, 0, 1, 0, 0, 0, 0, 0, 1, 1 ], 
  [ 0, 1, 0, 0, 1, 0, 0, 1, 0, 0, 0, 0, 1, 0, 1, 0, 1 ], 
  [ 0, 1, 1, 0, 0, 0, 2, 0, 0, 0, 0, 0, 0, 0, 1, 1, 0 ] ]
gap> C2:=dComp(C,C,-1);; # C2=C^(2)
gap> C4:=dComp(C2,C2,-1);; # C4=C^(4)
gap> C8:=dComp(C4,C4,-1);; # C8=C^(8)
gap> C16:=dComp(C8,C8,-1);; # C16=C^(16)
gap> C17:=dComp(C16,C,-1);; # C17=C^(17)
gap> P17:=CharacteristicPolynomial(C17);; # P17 is the priod polynomial for r=17
gap> R17:=RootsOfPolynomial(P17);; # roots of P17
gap> L17:=List(R17,x->17*x+1); # roots of the reduced period polynomial P17^*
[ 670357206530506901, 670088231006862759, 587253242462231659, 
  464742031061114921, 461908111585063663, 356621718684896633, 
  282238003107978403, 163898849457734107, 35922811461007315, 
  -197783211402587952, -205052440856501077, -243310155546790559, 
  -289516205265127375, -373934090375919493, -478731856802195967, 
  -651513206543247755, -1253189038565026183 ]
gap> X17:=L17/103;;
gap> List(X17,x->x mod 103); # X17[16]=-1 mod 103
[ 73, 31, 24, 94, 90, 22, 3, 95, 89, 42, 10, 80, 27, 69, 39, 102, 37 ]
gap> List(X17,x->x^2 mod 103);
[ 76, 34, 61, 81, 66, 72, 9, 64, 93, 13, 100, 14, 8, 23, 79, 1, 30 ]
gap> List(X17,x->x^3 mod 103);
[ 89, 24, 22, 95, 69, 39, 27, 3, 37, 31, 73, 90, 10, 42, 94, 102, 80 ]
gap> List(X17,x->x^17 mod 103); # X17[i] (i<>16) is of order 34 in F103^x
[ 102, 102, 102, 102, 102, 102, 102, 102, 102, 102, 
  102, 102, 102, 102, 102, 102, 102 ]
gap> L17[16]; # L17[16] is the exponential Gauss sum g_{17}(17)
-651513206543247755
gap> Factors(L17[16]);
[ -5, 7, 103, 172709, 1046412659 ]

gap> PrimitiveRootMod(191); # g=19
19
gap> C:=MultMat(19,191,19); # Multiplication matrix C1 for e=19, p=191, g=19
[ [ -10, -10, -10, -10, -10, -10, -9, -8, -8, -10, 
    -10, -8, -10, -10, -10, -10, -10, -10, -8 ],
  [ 0, 2, 0, 1, 2, 2, 0, 0, 0, 0, 1, 0, 0, 0, 1, 1, 0, 0, 0 ],
  [ 0, 0, 0, 0, 0, 1, 1, 1, 1, 0, 1, 0, 2, 0, 0, 1, 1, 0, 1 ],
  [ 0, 1, 0, 0, 0, 1, 0, 0, 0, 0, 2, 1, 0, 1, 0, 1, 0, 1, 2 ],
  [ 0, 2, 0, 0, 0, 1, 1, 1, 0, 0, 1, 1, 0, 0, 0, 0, 0, 1, 2 ],
  [ 0, 2, 1, 1, 1, 0, 1, 0, 0, 0, 1, 1, 0, 0, 1, 0, 0, 1, 0 ],
  [ 1, 0, 1, 0, 1, 1, 0, 0, 0, 1, 0, 0, 1, 1, 1, 1, 0, 1, 0 ],
  [ 2, 0, 1, 0, 1, 0, 0, 0, 0, 2, 0, 0, 0, 1, 0, 1, 2, 0, 0 ],
  [ 2, 0, 1, 0, 0, 0, 0, 0, 2, 0, 0, 1, 1, 1, 0, 0, 1, 1, 0 ],
  [ 0, 0, 0, 0, 0, 0, 1, 2, 0, 0, 1, 1, 2, 1, 1, 0, 0, 0, 1 ],
  [ 0, 1, 1, 2, 1, 1, 0, 0, 0, 1, 0, 0, 0, 0, 0, 0, 1, 2, 0 ],
  [ 2, 0, 0, 1, 1, 1, 0, 0, 1, 1, 0, 2, 0, 1, 0, 0, 0, 0, 0 ],
  [ 0, 0, 2, 0, 0, 0, 1, 0, 1, 2, 0, 0, 2, 0, 1, 0, 1, 0, 0 ],
  [ 0, 0, 0, 1, 0, 0, 1, 1, 1, 1, 0, 1, 0, 1, 0, 1, 0, 1, 1 ],
  [ 0, 1, 0, 0, 0, 1, 1, 0, 0, 1, 0, 0, 1, 0, 0, 2, 1, 1, 1 ],
  [ 0, 1, 1, 1, 0, 0, 1, 1, 0, 0, 0, 0, 0, 1, 2, 0, 2, 0, 0 ],
  [ 0, 0, 1, 0, 0, 0, 0, 2, 1, 0, 1, 0, 1, 0, 1, 2, 0, 1, 0 ],
  [ 0, 0, 0, 1, 1, 1, 1, 0, 1, 0, 2, 0, 0, 1, 1, 0, 1, 0, 0 ],
  [ 2, 0, 1, 2, 2, 0, 0, 0, 0, 1, 0, 0, 0, 1, 1, 0, 0, 0, 0 ] ]
gap> C2:=dComp(C,C,-1);; # C2=C^(2)
gap> C4:=dComp(C2,C2,-1);; # C4=C^(4)
gap> C8:=dComp(C4,C4,-1);; # C8=C^(8)
gap> C16:=dComp(C8,C8,-1);; # C16=C^(16)
gap> C3:=dComp(C2,C,-1);; # C3=C^(3)
gap> C19:=dComp(C16,C3,-1);; # C19=C^(19)
gap> P19:=CharacteristicPolynomial(C19);; # P19 is the priod polynomial for r=19
gap> R19:=RootsOfPolynomial(P19);; # roots of P19
gap> L19:=List(R19,x->19*x+1); # roots of the reduced period polynomial P19^*
[ 55891098112086637001228, 21343147495425176673226, 16127550524178031129657, 
  14355859672843887131634, 10195021892556248415182, 7777342710886644977131, 
  5776338119599847350627, 5080513863740739683465, 2801935824159299141695, 
  859413598509266105572, -1967831693815607448660, -2042500136091280335075, 
  -5599389538599795630810, -11060282774339943468556, -14117536712596171711328, 
  -19950229182831388897609, -27250892079645375357179, -28187266231514473770821, 
  -30032293464551740989379 ]
gap> X19:=List(L19,x->x/191);; 
gap> List(X19,x->x mod 191);
[ 58, 182, 157, 178, 94, 21, 126, 151, 142, 57, 
  99, 113, 146, 88, 112, 143, 105, 137, 183 ]
gap> List(X19,x->x^10 mod 191); # X19[9]^10=1 mod 191
[ 52, 136, 107, 121, 125, 30, 25, 154, 1, 32, 
  36, 69, 6, 160, 5, 150, 153, 177, 180 ]
gap> List(X19,x->x^2 mod 191);
[ 117, 81, 10, 169, 50, 59, 23, 72, 109, 2, 
  60, 163, 115, 104, 129, 12, 138, 51, 64 ]
gap> List(X19,x->x^5 mod 191); # X19[9] is of order 10 in F191^x
[ 166, 161, 38, 11, 70, 139, 186, 66, 190, 37, 
  185, 159, 31, 55, 14, 155, 41, 122, 84 ]
gap> List(X19,x->x^19 mod 191); # X19[i] (i<>9) is of order 190 in F191^x
[ 152, 152, 152, 152, 152, 152, 152, 152, 152, 152, 
  152, 152, 152, 152, 152, 152, 152, 152, 152 ]
gap> L19[9]; # L19[9] is the exponential Gauss sum g_{19}(19)
2801935824159299141695
gap> Factors(L19[9]);
[ 5, 191, 509, 26374987, 218546963 ]

gap> PrimitiveRootMod(47); # g=5
5
gap> C:=MultMat(23,47,5);  # Multiplication matrix C1 for e=23, p=47, g=5
[ [ -2, -2, -2, -2, -2, -2, -2, -2, -2, -2, -2, -2, 
    -2, -2, -2, -2, -2, -2, -1, -2, -2, -2, -2 ],
  [ 0, 0, 0, 0, 0, 0, 0, 0, 0, 0, 0, 0, 0, 1, 0, 1, 0, 0, 0, 0, 0, 0, 0 ],
  [ 0, 0, 0, 0, 0, 1, 1, 0, 0, 0, 0, 0, 0, 0, 0, 0, 0, 0, 0, 0, 0, 0, 0 ],
  [ 0, 0, 0, 0, 0, 0, 0, 0, 0, 0, 0, 0, 0, 0, 0, 0, 1, 0, 0, 0, 0, 1, 0 ],
  [ 0, 0, 0, 0, 0, 0, 0, 0, 0, 0, 0, 1, 0, 0, 0, 0, 0, 0, 0, 0, 0, 1, 0 ],
  [ 0, 0, 1, 0, 0, 1, 0, 0, 0, 0, 0, 0, 0, 0, 0, 0, 0, 0, 0, 0, 0, 0, 0 ],
  [ 0, 0, 1, 0, 0, 0, 0, 0, 0, 0, 0, 0, 0, 0, 1, 0, 0, 0, 0, 0, 0, 0, 0 ],
  [ 0, 0, 0, 0, 0, 0, 0, 0, 0, 0, 1, 0, 0, 0, 0, 0, 0, 0, 0, 1, 0, 0, 0 ],
  [ 0, 0, 0, 0, 0, 0, 0, 0, 0, 1, 0, 0, 0, 0, 0, 0, 0, 1, 0, 0, 0, 0, 0 ],
  [ 0, 0, 0, 0, 0, 0, 0, 0, 1, 0, 0, 0, 0, 0, 0, 1, 0, 0, 0, 0, 0, 0, 0 ],
  [ 0, 0, 0, 0, 0, 0, 0, 1, 0, 0, 0, 1, 0, 0, 0, 0, 0, 0, 0, 0, 0, 0, 0 ],
  [ 0, 0, 0, 0, 1, 0, 0, 0, 0, 0, 1, 0, 0, 0, 0, 0, 0, 0, 0, 0, 0, 0, 0 ],
  [ 0, 0, 0, 0, 0, 0, 0, 0, 0, 0, 0, 0, 0, 0, 0, 0, 1, 0, 0, 0, 0, 0, 1 ],
  [ 0, 1, 0, 0, 0, 0, 0, 0, 0, 0, 0, 0, 0, 0, 0, 0, 0, 0, 0, 0, 1, 0, 0 ],
  [ 0, 0, 0, 0, 0, 0, 1, 0, 0, 0, 0, 0, 0, 0, 0, 0, 0, 0, 0, 0, 0, 0, 1 ],
  [ 0, 1, 0, 0, 0, 0, 0, 0, 0, 1, 0, 0, 0, 0, 0, 0, 0, 0, 0, 0, 0, 0, 0 ],
  [ 0, 0, 0, 1, 0, 0, 0, 0, 0, 0, 0, 0, 1, 0, 0, 0, 0, 0, 0, 0, 0, 0, 0 ],
  [ 0, 0, 0, 0, 0, 0, 0, 0, 1, 0, 0, 0, 0, 0, 0, 0, 0, 0, 0, 1, 0, 0, 0 ],
  [ 1, 0, 0, 0, 0, 0, 0, 0, 0, 0, 0, 0, 0, 0, 0, 0, 0, 0, 0, 0, 1, 0, 0 ],
  [ 0, 0, 0, 0, 0, 0, 0, 1, 0, 0, 0, 0, 0, 0, 0, 0, 0, 1, 0, 0, 0, 0, 0 ],
  [ 0, 0, 0, 0, 0, 0, 0, 0, 0, 0, 0, 0, 0, 1, 0, 0, 0, 0, 1, 0, 0, 0, 0 ],
  [ 0, 0, 0, 1, 1, 0, 0, 0, 0, 0, 0, 0, 0, 0, 0, 0, 0, 0, 0, 0, 0, 0, 0 ],
  [ 0, 0, 0, 0, 0, 0, 0, 0, 0, 0, 0, 0, 1, 0, 1, 0, 0, 0, 0, 0, 0, 0, 0 ] ]
gap> C2:=dComp(C,C,-1);; # C2=C^(2)
gap> C4:=dComp(C2,C2,-1);; # C4=C^(4)
gap> C8:=dComp(C4,C4,-1);; # C8=C^(8)
gap> C16:=dComp(C8,C8,-1);; # C16=C^(16)
gap> C3:=dComp(C2,C,-1);; # C3=C^(3)
gap> C7:=dComp(C4,C3,-1);; # C7=C^(7)
gap> C23:=dComp(C16,C7,-1);; # C23=C^(23)
gap> P23:=CharacteristicPolynomial(C23);; # P23 is the priod polynomial for r=23
gap> R23:=RootsOfPolynomial(P23);; # roots of P23
gap> L23:=List(R23,x->23*x+1); # roots of the reduced period polynomial P23^*
[ 142339874433137221525, 118065170266710759348, 90401156916499269233, 
  55373954393947818396, 55099193646218848063, 42654144441633168738, 
  42378310496086559486, 36268843595424974262, 35660322726333362220, 
  34760976326466677323, 28446187386897694871, 17050560055492972666, 
  -492643134044787602, -9055501540645768832, -16107397702852877550, 
  -31331987537967805455, -36858108220907188977, -38922282154313258582, 
  -39922556198217904917, -67269172064831016965, -90222434992270940059, 
  -151942428479066503710, -216374182659731273482 ]
gap> X23:=L23/47;;
gap> List(X23,x->x mod 47); # X23[13]=1 mod 47
[ 9, 32, 24, 4, 17, 12, 8, 2, 3, 37, 14, 18, 
  1, 21, 36, 16, 25, 28, 42, 27, 34, 6, 7 ]
gap> List(X23,x->x^2 mod 47);
[ 34, 37, 12, 16, 7, 3, 17, 4, 9, 6, 8, 42, 
  1, 18, 27, 21, 14, 32, 25, 24, 28, 36, 2 ]
gap> List(X23,x->x^23 mod 47); # X23[i] (i<>13) is of order 23 in F47^x
[ 1, 1, 1, 1, 1, 1, 1, 1, 1, 1, 1, 1, 1, 1, 1, 1, 1, 1, 1, 1, 1, 1, 1 ]
gap> L23[13]; # L23[13] is the exponential Gauss sum g_{23}(23)
-492643134044787602
gap> Factors(L23[13]);
[ -2, 17, 43, 47, 7169472509893 ]
\end{verbatim}

\begin{acknowledgment}
The authors would like to thank one of the referees 
who read the manuscript very carefully and gave helpful suggestions. 
\end{acknowledgment}

%%%%%%%%%%%%%%%%%%%%%%%%%%%%%%%%%%%%

\end{document}